\documentclass[11pt]{amsart}

\usepackage{amsmath, amscd, amssymb}
\usepackage[frame,cmtip,arrow,matrix,line,graph,curve]{xy}
\usepackage{graphpap, color}
\usepackage[mathscr]{eucal}
\usepackage{cancel}
\usepackage{verbatim}
\usepackage{adjustbox}
\usepackage{tikz}
\usepackage{float}
\usepackage{booktabs}
\usepackage{multirow}
\usepackage{tabularx}

\usepackage{hyperref}
\numberwithin{equation}{section}

\newtheorem{theorem}{Theorem}[section]
\newtheorem{corollary}[theorem]{Corollary}
\newtheorem{proposition}[theorem]{Proposition}
\newtheorem{lemma}[theorem]{Lemma}

\newtheorem{conjecture}[theorem]{Conjecture}

\theoremstyle{definition}
\newtheorem{definition}[theorem]{Definition}
\newtheorem{question}[theorem]{Question}
\newtheorem{example}[theorem]{Example}
\newtheorem{remark}[theorem]{Remark}

\def\red{\textcolor{red}}

\newcommand{\Z}{\mathbb{Z}}

\newcommand{\C}{\mathbb{C}}

\newcommand{\PP}{\mathbb{P}}

\def\CC{\mathbb{C}}

\def\QQ{\mathbb{Q}}

\def\sO{{\mathscr O}}

\newcommand{\cal}{\mathcal}
\def\cA{{\cal A}}

\def\cF{{\cal F}}

\def\cH{{\cal H}}

\def\cM{{\cal M}}

\def\cV{{\cal V}}

\def\cY{{\cal Y}}

\def\tX{{\widetilde{X}}}

\def\tP{\widetilde{P}}

\def\hbar{\overline{h}}

\def\M{\overline{\cM}}

\def\dim{\mathrm{dim} }

\def\Pic{\mathrm{Pic} }

\def\ch{\mathrm{ch} }

\def\and{\quad{\rm and}\quad}
\def\lra{\longrightarrow }
\def\mapright#1{\,\smash{\mathop{\lra}\limits^{#1}}\,}

\def\beq{\begin{equation}}
\def\eeq{\end{equation}}
\def\ben{\begin{enumerate}}
\def\een{\end{enumerate}}

\def\and{\quad\text{and}\quad}

\def\O{\mathscr{O}}

\def\Fl{\mathrm{Fl}}

\def\Fl{\mathrm{Fl}}
\def\E{\mathrm{E}}
\def\V{\mathrm{V}}
\def\G{\Gamma}
\def\k{\kappa}

\def\csf{\mathsf{csf}}
\def\asc{\mathsf{asc}}
\def\span{\mathrm{span} }

\newcommand{\N}{\mathbb{N}}

\title[Birational geometry of Hessenberg varieties]{Birational geometry of generalized Hessenberg varieties and the generalized Shareshian-Wachs conjecture}
\date{February 14, 2024.}

\author{Young-Hoon Kiem}
\address{School of Mathematics, Korea Institute for Advanced Study, 85 Hoegiro, Dongdaemun-gu, Seoul 02455, Korea}
\email{kiem@kias.re.kr}

\author{Donggun Lee}
\address{Center for Complex Geometry, Institute for Basic Science (IBS), 55 Expo-ro, Yuseong-gu, Daejeon 34126, Korea}
\email{dglee@ibs.re.kr}

\thanks{YHK was partially supported by Korea NRF grant 2021R1F1A1046556. DL was supported by the Institute for Basic Science (IBS-R032-D1).}
\begin{document}

\begin{abstract}
We introduce generalized Hessenberg varieties and establish basic facts. We show that the Tymoczko action of the symmetric group $S_n$ on the cohomology of Hessenberg varieties extends to generalized Hessenberg varieties and that natural morphisms among them preserve the action. 
By analyzing natural morphisms and birational maps among generalized Hessenberg varieties, we give an elementary proof of the Shareshian-Wachs conjecture. Moreover we present a natural generalization of the Shareshian-Wachs conjecture that involves generalized Hessenberg varieties and provide an elementary proof. As a byproduct, we propose a generalized Stanley-Stembridge conjecture for \emph{weighted} graphs. Our investigation into the birational geometry of generalized Hessenberg varieties enables us to modify them into much simpler varieties like projective spaces or permutohedral varieties by explicit sequences of blowups or projective bundle maps. Using this, we provide two algorithms to compute the $S_n$-representations on the cohomology of generalized Hessenberg varieties. As an application, we compute representations on the low degree cohomology of some Hessenberg varieties. 
\end{abstract}

\maketitle
%\tableofcontents

\section{Introduction}\label{38}

Hessenberg varieties are closed subvarieties of the variety $\Fl(n)$ of flags in 
%$$V_1\subset V_2\subset\cdots\subset V_{n-1}$$
$\CC^n$ with many interesting properties \cite{DMPS}.
The Shareshian-Wachs conjecture (Theorem \ref{62}) formulated in \cite{SW} and proved in \cite{BC, GP2} connects the cohomology of Hessenberg varieties with the chromatic quasi-symmetric functions of the associated graphs, which are refinements of the chromatic polynomials defined and studied by G. Birkhoff and H. Whitney in the early twentieth century towards the celebrated four color problem. 

From geometric point of view, Hessenberg varieties have rich structure with lots of morphisms and birational maps among them which are quite useful in the study of their cohomology. 
However, for a full-fledged investigation into the birational geometry of Hessenberg varieties, it seems inevitable to include related algebraic varieties that appear for instance from contractions.

In this paper, we introduce generalized Hessenberg varieties (Definition \ref{5} and Theorem \ref{4}) and show that the torus action and the Tymoczko action \cite{Tym} extend  to generalized Hessenberg varieties (Theorem \ref{17}). Then we prove  that natural morphisms among them preserve the Tymoczko action of $S_n$ (Propositions \ref{21}, \ref{24}, \ref{33}). 
Using these facts, we compare the cohomology of generalized Hessenberg varieties which are related by certain forgetful morphisms (Theorem \ref{26}). 

Based on these basics on generalized Hessenberg varieties, we analyze natural morphisms and birational maps which are decomposed into blowups and projective bundle maps (Theorem \ref{42}). Our analysis leads us to a short elementary proof of the Shareshian-Wachs conjecture (Theorem \ref{62}) by establishing the modular law (Proposition \ref{64}). 
On the other hand, the birational geometry of generalized Hessenberg varieties enables us to modify them into simpler ones and to give formulas that compare the $S_n$-representations on the cohomology of generalized Hessenberg varieties (Proposition \ref{73} and Corollary \ref{48}). 

The Shareshian-Wachs conjecture involves the ordinary Hessenberg varieties in \cite{DMPS} only and it is natural to ask for a generalization that includes the cohomology of generalized Hessenberg varieties. We formulate a natural generalization of the chromatic quasi-symmetric functions for \emph{weighted} graphs (Definition \ref{82}) and an obvious generalization of the Shareshian-Wachs conjecture that relates the cohomology of generalized Hessenberg varieties with the chromatic quasi-symmetric functions of corresponding weighted graphs. 
We prove this generalized conjecture by using relations among generalized Hessenberg varieties (Theorem \ref{82}). We also propose a generalized Stanley-Stembridge conjecture for weighted graphs (Question \ref{81}). 

Since our analysis on the birational geometry enables us to modify generalized Hessenberg varieties into permutohedral varieties or projective spaces through well understood operations, we can use it to compute the cohomology of generalized Hessenberg varieties as $S_n$-representations. 
We provide two algorithms which can be easily implemented on computer programs like Sage (Propositions \ref{101} and \ref{105}) and give a formula for Hessenberg varieties (Theorem \ref{102}). 

As an application of our algorithms, we compute the low degree cohomology of the Hessenberg varieties $X_{h_k}$ associated to the Hessenberg functions $h_k$ defined by $h_k(i)=\min\{i+k,n\}$. 
We provide an alternative proof for some previously known results in \cite{AMS} on the cohomology $H^{\le 2k}(X_{h_k})$ of degrees $\le 2k$ based on our algorithms (Theorems~\ref{114} and \ref{113}) and compute $H^{2k+2}(X_{h_k})$ explicitly (Theorem~\ref{111}), thus establishing the Stanley-Stembridge conjecture for $H^{\leq 2k+2}(X_{h_k})$.

All the varieties in the paper are defined over the complex number field $\CC$. We denote by $\N$ the set of positive integers.

\noindent{\textbf{Acknowledgement.} We thank Jaehyun Hong and Jaeseong Oh for enlightening discussions. We also thank Timothy Chow and Eric Sommers for useful comments.

\bigskip

\section{Generalized Hessenberg varieties}\label{11}
In this section, we introduce \emph{generalized Hessenberg varieties} 
and prove that the Tymoczko action of $S_n$ extends to the cohomology of generalized Hessenberg varieties. Furthermore, we prove several natural maps among the cohomology groups of generalized Hessenberg varieties are $S_n$-equivariant.

\subsection{Definition and basic properties}\label{12}

In this paper, we only deal with regular semisimple Hessenberg varieties of type A. 
So we drop the phrases ``regular semisimple'' and ``type A'' for simplicity.

Let $x$ be a diagonal $n\times n$ matrix with \emph{distinct nonzero} complex entries. The choice of $x$
does not play a role and we fix $x$ once and for all throughout this paper. 
 For any positive integer $m$, we write $[m]=\{1,2,\cdots,m\}$. 
\begin{definition}\label{5} 
Let $I$ be a subset of $[n-1]$. 
\begin{enumerate}
    \item A \emph{generalized Hessenberg function} is a function
    \[h:I\cup \{n\}\longrightarrow I\cup \{n\}\]
    satisfying \begin{enumerate}\item $h(i)\geq i$ for all $i\in I\cup\{n\}$ and  
    \item $h(i)\le h(j)$ whenever $i<j$ in $I$.\end{enumerate} 
    We let $\cH_{I,n}$ denote the set of all generalized Hessenberg functions $h:I\cup \{n\}\to I\cup\{n\}$. When $I=[n-1]$, we let $\cH_n:=\cH_{[n-1],n}$ and call an $h\in \cH_n$ an ordinary Hessenberg function. When $I=[r]$ for $1\le r\le n$, we write $\cH_{r,n}=\cH_{[r],n}$. 
    \item A \emph{generalized Hessenberg variety} corresponding to the generalized Hessenberg function $h$ is the closed subvariety
    \beq \label{7}
    X_h=\left\{(V_i)_{i\in I}\in \Fl_I(n)~:~xV_i\subset V_{h(i)} ~\text{ for all }~ i\in I\right\}\eeq
    of the partial flag variety 
    \beq\label{8}\Fl_I(n)=\left\{(V_i)_{i\in I}~:~V_i\subset V_{j}\subset \CC^n ~\text{ for }~ i\le j\text{ in }I, ~\dim V_i=i\right\} 
    \eeq 
    of subspaces of $\CC^n$. We write $\Fl_r(n)=\Fl_{[r]}(n)$ for $r<n$. We let $\Fl(n)=\Fl_{[n-1]}(n)$  denote the variety of full flags in $\CC^n$.  When $h\in \cH_n$, we say that $X_h$ is an ordinary Hessenberg variety. 
\end{enumerate}
\end{definition}
When $I=\emptyset$, there is a unique $h:\{n\}\to \{n\}$ in $\cH_{I,n}$ for which $X_h$ is just a point. 
\begin{example}\label{6a}
(1) If $I=\{1\}$ and $h(1)=n$, then $X_h=\PP^{n-1}$. More generally, if $h(i)=n$ for all $i\in I$, then $X_h=\Fl_I(n)$ is the partial flag variety. 

(2) If $I=\{i\}$ and $h(i)=i$, then $X_h$ consists of $\binom{n}{i}$ points representing $i$-dimensional subspaces spanned by $i$ standard basis vectors. We call these $i$-dimensional  \emph{coordinate subspaces}. 

(3) When $I=[n-1]$, $X_h$ are the usual Hessenberg varieties defined and studied in \cite{DMPS}. 
\end{example}

\medskip

There are many ways to see that the generalized Hessenberg varieties $X_h$ are algebraic varieties. Here is one way which will be useful in \S\ref{36}. Let 
\beq\label{27}
I=\{i_1,i_2,\cdots, i_r\}\subset [n-1], \quad 0=i_0<i_1<i_2<\cdots <i_r<i_{r+1}=n.
\eeq
The partial flag variety is equipped with tautological bundles
$$0=\cV_{i_0}\subset \cV_{i_1}\subset \cV_{i_2}\subset \cdots \subset \cV_{i_r}\subset \cV_{i_{r+1}}=\sO^{\oplus n}_{\Fl_I(n)}.$$
The linear operator $x:\CC^n\to\CC^n$ induces a section 
$$x_1\in H^0(\Fl_I(n), \cV_{i_1}^*\otimes \sO^{\oplus n}_{\Fl_I(n)}/\cV_{h(i_1)})=\mathrm{Hom}_{\Fl_I(n)}(\cV_{i_1}, \sO^{\oplus n}_{\Fl_I(n)}/\cV_{h(i_1)})$$
of the vector bundle $\cV_{i_1}^*\otimes \sO^{\oplus n}_{\Fl_I(n)}/\cV_{h(i_1)}$ of rank $i_1(n-h(i_1))$ by the composition
$$\cV_{i_1}\hookrightarrow \sO^{\oplus n} \mapright{x} \sO^{\oplus n}\lra \sO^{\oplus n}/\cV_{h(i_1)}.$$
The zero locus $Z(x_1)\subset \Fl_I(n)$ of $x_1$ parameterizes partial flags $(V_i)_{i\in I}$ satisfying $xV_{i_1}\subset V_{h(i_1)}$. 

Over $Z(x_1)$, $x$ induces a section 
$$x_2\in H^0(Z(x_1), (\cV_{i_2}/\cV_{i_1})^*\otimes \sO^{\oplus n}/\cV_{h(i_2)})$$
of the vector bundle $(\cV_{i_2}/\cV_{i_1})^*\otimes \sO^{\oplus n}/\cV_{h(i_2)}|_{Z(x_1)}$ of rank $(i_2-i_1)(n-h(i_2))$ by the composition 
$$\cV_{i_2}/\cV_{i_1}\hookrightarrow \sO^{\oplus n}/\cV_{i_1} \mapright{x} \sO^{\oplus n}/\cV_{h(i_1)}\lra \sO^{\oplus n}/\cV_{h(i_2)}.$$
The zero locus $Z(x_1,x_2)\subset Z(x_1)$ of $x_2$ parameterizes partial flags $(V_i)_{i\in I}$ satisfying $xV_{i_1}\subset V_{h(i_1)}$ and $xV_{i_2}\subset V_{h(i_2)}$. 
Continuing this way, we find that 
$$X_h=Z(x_1,\cdots, x_r)\subset \Fl_I(n)$$
is a closed subvariety of the smooth projective variety $\Fl_I(n)$. 

It is well known that the dimension of $\Fl_I(n)$ is 
$$\sum_{k=1}^r (i_k-i_{k-1})(n-i_k)$$
while the sum of the ranks of $(\cV_{i_k}/\cV_{i_{k-1}})^*\otimes \sO^{\oplus n}/\cV_{h(i_k)}$ is
$$\sum_{k=1}^r (i_k-i_{k-1})(n-h(i_k)).$$
Therefore the expected dimension of $X_h$ is 
$$\sum_{k=1}^r (i_k-i_{k-1})((n-i_k)-(n-h(i_k)))=\sum_{k=1}^r(i_k-i_{k-1})(h(i_k)-i_k)$$
which is the actual dimension of $X_h$ by the following. 

\begin{theorem}\label{4}
(1) For $h\in \cH_{I,n}$ with $I$ in \eqref{27}, the Hessenberg variety $X_h$ is a smooth projective variety of dimension 
\beq\label{1} \sum_{k=1}^r(i_k-i_{k-1})(h(i_k)-i_k).\eeq

(2) If $h(i)>i$ for all $i\in I$, $X_h$ is irreducible. If $h(i)=i$ for some $i\in I$, $X_h$ is isomorphic to the disjoint union of $\binom{n}{i}$ copies of $X_{h'}\times X_{h''}$ where 
$$I'=\{i'\in I~:~i'<i\}, \quad h'=h|_{I'\cup \{i\}}\in \cH_{I',i},$$
$$I''=\{i''-i~:~ i''>i, i''\in I\}, \quad h''(i''-i)=h(i'')-i, ~ h''\in \cH_{I'',n-i}.$$

(3) The cycle class map $A^*(X_h)\to H^*(X_h)$ from the Chow ring to the cohomology ring is an isomorphism. 
\end{theorem}

The proof is obtained by comparing $X_h$ with an ordinary Hessenberg variety $X_{\tilde{h}}$ for a Hessenberg function $\tilde{h}$ defined below. In fact, there is a canonical smooth projective morphism $X_{\tilde{h}}\to X_h$.
\begin{definition}
	For $h\in \cH_{I,n}$ with $I$ in \eqref{27}, define $\tilde{h}\in \cH_n$ by
\beq\label{28}
\tilde{h}:[n]\lra [n], \quad \tilde{h}(i)=h(i_k) \ \ \text{for }\ i_{k-1}<i\le i_k.
\eeq
\end{definition}

By definition, we have the Cartesian diagram
\beq\label{2}
\xymatrix{
X_{\tilde{h}}\ar[r] \ar[d] & \Fl_{[n-1]}(n)\ar[d]\\
X_h\ar[r] &\Fl_I(n)
}\eeq
whose right vertical arrow is the forgetful morphism which is smooth with fibers 
$$\prod_{k=1}^{r+1}\Fl(i_k-i_{k-1})$$ because 
$\Fl_{[n-1]}(n)$ is in fact the fiber product 
$$\Fl(\cV_{i_1})\times_{\Fl_I(n)}\Fl(\cV_{i_2}/\cV_{i_{1}})\times_{\Fl_I(n)}\cdots \times_{\Fl_I(n)} \Fl(\cV_{i_{r+1}}/\cV_{i_{r}})$$
where $\Fl(\cV_{i_k}/\cV_{i_{k-1}})$ denotes the flag bundle parameterizing full flags in $\cV_{i_k}/\cV_{i_{k-1}}|_\xi$ for $\xi\in \Fl_I(n)$. 

\medskip

\noindent\emph{Proof of Theorem~\ref{4}}.
By \cite[Theorem 6]{DMPS}, $X_{\tilde{h}}$ is a smooth projective variety of dimension $\sum_{i=1}^n(\tilde{h}(i)-i)$. Hence $X_h$ is smooth of dimension
\beq\label{3} \sum_{i=1}^n(\tilde{h}(i)-i)-\sum_{k=1}^{r+1}\dim \,\Fl(i_k-i_{k-1}).\eeq
It is elementary to see that \eqref{3} equals \eqref{1} by $\dim\, \Fl(m)=m(m-1)/2$. 

By \eqref{2}, (2) and (3) follow from the corresponding statements for $X_{\tilde{h}}$ in \cite{DMPS}. 
\hfill $\square$

\medskip

Another way to see Theorem \ref{4} is as follows. Let $G=\mathrm{GL}_n(\CC)$ and $B$ be the subgroup of upper triangular matrices. By \cite{DMPS}, the top row in \eqref{2} is 
$$X_{\tilde{h}}=G_{\tilde{h}}/B\hookrightarrow G/B=\Fl_{[n-1]}(n)=\Fl(n)$$
where $G_{\tilde{h}}=\{g\in G\,|\, (g^{-1}xg)_{ij}=0\text{ if }i>\tilde{h}(j)\}.$
Here $A_{ij}$ denotes the $(i,j)$-th entry of a matrix $A$. 
Let 
$$P_I=\{g\in G\,|\, g_{ij}=0\text{ if }i_{k-1}<j\le i_k<i\}.$$
Then $P_I\subset G_{\tilde{h}}$ and 
$$X_h=G_{\tilde{h}}/P_I\hookrightarrow G/P_I=\Fl_I(n)$$
is the bottom row in \eqref{2}. The vertical arrows in \eqref{2} have fibers 
$$P_I/B\cong \prod_{k=1}^{r+1}\Fl(i_k-i_{k-1}).$$

\begin{remark}
	Generalized Hessenberg varieties are equal to a certain class of parabolic Lusztig varieties recently defined in \cite{AN3}.
	To be precise, for a generalized Hessenberg function $h:I\cup \{n\}\to I\cup\{n\}$, let $w\in S_n$ be the codominant permutation associated to the Hessenberg function $\tilde{h}$ defined by \eqref{28}, i.e. the maximum element of $\{w'\in S_n:w'(i)\leq \tilde{h}(i) \text{ for all }i\}$ with respect to the Bruhat order.
	Let $J=[n-1]-I$. Then, $X_h$ is equal to the parabolic Lusztig variety $\cY_{w_0,J}(x)$, as a subvariety of $\Fl_I(n)$, for the minimum element $w_0$ of $(S_n)_Jw$ with respect to the Bruhat order, where $(S_n)_J$ denotes the subgroup of $S_n$ generated by simple transpositions $(j,j+1)\in S_n$ for $j\in J$ (see \cite[Proposition 6.1]{AN3}).
\end{remark}

\subsection{Torus action and the Tymoczko action}\label{13}
 In this subsection, we show that the torus action and Tymoczko's dot action in \cite{Tym} extend to the setting of generalized Hessenberg varieties.

Let $T=(\CC^*)^n\subset G=\mathrm{GL}_n(\CC)$ denote the group of diagonal invertible matrices. Let $h\in \cH_{I,n}$. 
Then $T$ acts on the generalized Hessenberg variety $X_h$ by sending
$(V_i)_{i\in I}$ to $(gV_i)_{i\in I}$ for $g\in T$. 
As $x, g\in T$  commute, this action is well defined. 

It is straightforward to see that the $T$-fixed point set $X_h^T$ in $X_h$ is 
\beq\label{14} X_h^T=S_nP_I/P_I\subset G_{\tilde{h}}/P_I=X_h\eeq
where we have identified the symmetric group $S_n$ with the group of permutation matrices. 
Hence $X_h^T$ is finite and has an action of $S_n$ by left multiplication. 
Indeed, given $\sigma\in S_n$, we have a fixed point $(V_i)_{i\in I}\in X_h^T$  defined by
\beq\label{15} V_i=\span\{e_{\sigma(1)}, e_{\sigma(2)},\cdots,e_{\sigma(i)}\}\eeq
where $e_1,\cdots, e_n$ are the standard basis vectors in $\CC^n$. 
In particular, we have a surjective $S_n$-equivariant map 
$$v:S_n\lra S_nP_I/P_I=X_h^T.$$

It is also straightforward that a 1-dimensional $T$-orbit in $X_h$ is the orbit of $(V_i)_{i\in I}$ such that for some pair $j<k$ in $I$ with $k\leq h(j)$ and for some $\sigma\in S_n$, $V_i$ is defined as in \eqref{15} if $i< j$ or $i\ge k$ and 
\beq\label{16}
V_i=\span\{e_{\sigma(1)}, e_{\sigma(2)},\cdots,e_{\sigma(j-1)},  e_{\sigma(j)}+e_{\sigma(k)}, e_{\sigma(j+1)}, \cdots, e_{\sigma(i)}\}\eeq
for $j\le i<k$. 
In particular, by \cite{GKM}, the $T$-equivariant cohomology ring of $X_h$ is 
\beq\label{19}
H^*_T(X_h)=\{(p_v)_{v\in X_h^T}\,|\, t_{\sigma(j)}-t_{\sigma(k)}\text{ divides }p_{v(\sigma)}-p_{v(\sigma\circ(j,k))}\text{ for }\sigma\in S_n\}
\eeq
where $p_v\in \QQ[t_1,\cdots,t_n]$ are polynomials in $n$ variables and $(j,k)\in S_n$ denotes the transposition interchanging $j$ and $k$. 

Tymoczko's dot action introduced in \cite{Tym} for the cohomology of ordinary Hessenberg varieties with $I=[n-1]$ now extends to the cohomology of generalized Hessenberg varieties without change as follows. 
For $\mu\in S_n$ and $p=(p_v)_{v\in X_h^T}\in H^*_T(X_h)$, let 
\beq\label{18} (\mu\cdot p)_v=\mu\cdot p_{\mu^{-1}(v)}\eeq
where $\mu\cdot f(t_1,\cdots,t_n)=f(t_{\mu(1)},\cdots, t_{\mu(n)})$ for $f\in \QQ[t_1,\cdots, t_n]$. 

\begin{theorem}\label{17}
The Tymoczko action of $S_n$ on $H^*_T(X_h)$ is well defined. It induces an $S_n$-action on $H^*(X_h)$ by the isomorphism
\beq\label{20} H^*(X_h)\cong H^*_T(X_h)/\mathbf{m} H^*_T(X_h)\eeq
where $\mathbf{m}=(t_1,\cdots,t_n)\subset \QQ[t_1,\cdots,t_n]$ is the maximal homogeneous ideal. 
\end{theorem}
\begin{proof}
Let $p=(p_v)_{v\in X_h^T}\in H^*_T(X_h)$. Let $\sigma\in S_n$ and $j\in I$. 
Consider the 1-dimensional orbit defined by \eqref{15} for $V_i$ with $i< j$ or $i\ge k$ and \eqref{16} for $V_i$ for $j\le i<k.$  
Then $t_{\sigma(j)}-t_{\sigma(k)}$ divides $p_{v(\sigma)}-p_{v(\sigma\circ(j,k))}$ for any $\sigma\in S_n$.
Hence, for $\mu\in S_n$, 
$$(\mu\cdot p)_{v(\sigma)}-(\mu\cdot p)_{v(\sigma\circ(j,k))}=
\mu\cdot \left( p_{v(\mu^{-1}\sigma)}-p_{v(\mu^{-1}\sigma\circ(j,k))} \right)$$
is divisible by 
$$\mu\cdot (t_{\mu^{-1}\sigma(j)}-t_{\mu^{-1}\sigma(k)})=t_{\sigma(j)}-t_{\sigma(k)}.$$
Therefore, $\mu\cdot p\in H^*_T(X_h)$ and the action \eqref{18} is well defined on \eqref{19}.

Since the Tymoczko action preserves the degrees of polynomials, the $S_n$-action on $H^*_T(X_h)$ induces an $S_n$-action on $H^*(X_h)\cong A^*(X_h)$ by \eqref{20}.
\end{proof}

\begin{remark}\label{25}
By Theorem \ref{4} (3), for any $k\ge 0$,
\beq \label{29}
H^{2k}(X_h)\cong A^k(X_h),\quad H^{2k+1}(X_h)=0.
\eeq
%$$H^{2k}(X_h)\cong A^k(X_h),\quad H^{2k+1}(X_h)=0.$$
For notational convenience, we use the Chow ring $A^*(X_h)$ instead of the cohomology ring $H^*(X_h)$ from now on. Note that \eqref{29} also follows from the Bia\l ynicki-Birula decomposition of $X_h$ by affine spaces as in \cite[III]{DMPS}.
\end{remark}

\begin{definition}\label{23}
For $h\in \cH_{I,n}$, the \emph{$S_n$-equivariant Poincar\'e polynomial} of $X_h$ is defined as 
\beq\label{23a} \cF(h)=\sum_{k\geq 0}\ch\left(A^k(X_h)\right)q^k \in \Lambda^n[q]\eeq
where $\ch$ denotes the Frobenius characteristic and $\Lambda^n$ denotes the group of symmetric functions of degree $n$. 
\end{definition}

For example, when $X_h=\PP^{n-1}$ as in Example \ref{6a} (1), \eqref{23a} is 
\beq\label{23b}
\cF(h)=[n]_q:=\frac{1-q^n}{1-q}=1+q+\cdots+q^{n-1}\eeq
because $S_n$ acts trivially on the cohomology group.
Here we suppressed the Frobenius characteristic $\mathsf{h}_n$ of the trivial representation of $S_n$, to simplify the notation.  
When $X_h=\Fl(n)$ is the full flag variety, \eqref{23a} is 
\beq\label{23c}
\cF(h)=[n]_q!:=[n]_q[n-1]_q\cdots [1]_q
\eeq
because $\Fl(n)$ is $Y_{n-1}$ in the sequence of projective bundles
$$Y_1=\PP^{n-1}, ~ Y_k=\PP_{Y_{k-1}} (\sO_{Y_{k-1}}^{\oplus n}/\cV_{k-1}), 
~ 1< k<n,$$
where $\cV_k$ is the vector bundle of rank $k$ obtained as the inverse image of $\sO_{Y_k}(-1)$
by the quotient $\sO^{\oplus n}_{Y_k}\to \sO^{\oplus n}_{Y_k}/\cV_{k-1}|_{Y_k}$, with $\cV_1=\sO_{Y_1}(-1)$. 
It is well known (cf. \cite{Tym}) that the Tymoczko action is trivial in this case. 
One can also immediately deduce \eqref{23c} from Corollary \ref{31} below by choosing $I=\emptyset$. 

The goal of this paper is to find $\cF(h)$ by the birational geometry of Hessenberg varieties.

\subsection{Comparison maps and equivariance} \label{22}
There are many morphisms among generalized Hessenberg varieties.
In this subsection, we see that the induced maps on cohomology are often $S_n$-equivariant. 

\begin{proposition}\label{21}
Let $h\in \cH_{I,n}$, $h'\in\cH_{I',n}$ and $T=(\CC^*)^n$. 
Let  $$f:X_{h'}\lra X_h$$ be a $T$-equivariant morphism whose restriction $X_{h'}^T\xrightarrow{}X_h^T$ to the fixed point loci is $S_n$-equivariant via \eqref{14}.  
Then the pullback
 \beq\label{21a} f^*:A^*(X_h)\longrightarrow A^*(X_{h'})\eeq
is $S_n$-equivariant.
\end{proposition}
\begin{proof} As $f$ is $T$-equivariant, we have a commutative diagram
\[\xymatrix{
X_h & X_{h'}\ar[l]_f\\
X_h^T\ar@{^(->}[u] & X^T_{h'}\ar[l]\ar@{^(->}[u]
}\]
which induces the commutative diagram
\[\xymatrix{
A^*_T(X_h)\ar[r]^{f^*}\ar@{^(->}[d] & A^*_T(X_{h'})\ar@{^(->}[d]\\
A^*_T(X_h^T)\ar[r] & A^*_T(X^T_{h'})
}.\]
By \eqref{18}, the bottom arrow and the vertical arrows are $S_n$-equivariant. Hence the top arrow is $S_n$-equivariant and by \eqref{20}, \eqref{21a} is $S_n$-equivariant as well.
\end{proof}

Recall that we have tautological bundles $(\cV_i)_{i\in I}$ on $X_h$ coming from those on the partial flag variety $\Fl_I(n)$. 
\begin{proposition}\label{24}
For $i<j$ in $I$ and 
$0\le d\le j-i$, the $d$-th Chern class
\[c_d(\cV_j/\cV_i): A^*(X_h)\longrightarrow A^{*+d}(X_h), \quad \xi\mapsto \xi\cup c_d(\cV_j/\cV_i)\]
is $S_n$-equivariant.
\end{proposition}
\begin{proof}
As the Chern classes commute with restrictions, we have a commutative diagram
\[\xymatrix{
A^*_T(X_h)\ar[rr]^{c_d(\cV_j/\cV_i)}\ar@{^(->}[d] && A^{*+d}_T(X_{h})\ar@{^(->}[d]\\
A^*_T(X_h^T)\ar[rr]^{c_d(\cV_j/\cV_i)|_{X_h^T}} && A^{*+d}_T(X^T_{h})
}.\]
Since the bottom and the vertical arrows are $S_n$-equivariant, so is the top arrow and \eqref{20} implies the proposition.
\end{proof}

If $I'\subset I$ and $h\in \cH_{I,n}$ satisfies $h':=h'|_{I'\cup\{n\}}\in \cH_{I',n}$, 
then we have the forgetful morphism
\beq\label{24a} \rho_{h,h'}:X_{h}\lra X_{h'}\eeq
which forgets $V_j$ for $j\in I-I'$. 
\begin{theorem}\label{26}
Let $h\in \cH_{I,n}$ with $I$ in the form of \eqref{27}. Suppose there is a $1\le k\le r$ such that  $h(i_k)=h(i_{k+1})$ and $h^{-1}(i_k)=\emptyset$. Let $I'=I-\{i_k\}$ and $h'=h|_{I'\cup\{n\}}\in \cH_{I',n}$. Then the morphism $\rho_{h,h'}$ which forgets $V_{i_k}$ is the  
Grassmannian bundle 
\beq\label{30b}\mathrm{Gr}(i_k-i_{k-1}, \cV_{i_{k+1}}/\cV_{i_{k-1}})\lra X_{h'}\eeq
and we have 
\beq\label{30a} \cF(h)=\cF(h')\frac{[i_{k+1}-i_{k-1}]_q!}{[i_{k}-i_{k-1}]_q![i_{k+1}-i_{k}]_q!}.\eeq

In particular, if $i_k=i_{k-1}+1$, $\rho_{h,h'}$ is the 
projective bundle 
$$\PP(\cV_{i_{k+1}}/\cV_{i_{k-1}})\lra X_{h'}.$$ 
In this case, the $S_n$-equivariant Poincar\'e polynomials satisfy
\beq\label{30} \cF(h)=\cF(h')\, [i_{k+1}-i_{k-1}]_q.\eeq 
\end{theorem}
\begin{proof}
Under our assumptions, it is easy to see that $X_h$ equals the fiber product  
$$X_h=X_{h'}\times_{\Fl_{I'}(n)} \Fl_I(n)$$
As the forgetful morphism 
$\Fl_I(n)\to \Fl_{I'}(n)$ is a Grassmannian bundle whose restriction to $X_{h'}$ is 
$\mathrm{Gr}(i_k-i_{k-1}, \cV_{i_{k+1}}/\cV_{i_{k-1}})\to X_{h'}$,
we find that $\rho_{h,h'}$ is \eqref{30b}.

By the Leray spectral sequence and Deligne's degeneration criterion in 
\cite[p.606]{GrHa}, we have an isomorphism
$$A^*(X_h)\cong A^*(X_{h'})\otimes A^*(\mathrm{Gr}(i_k-i_{k-1},i_{k+1}-i_{k-1})).$$
As the cohomology ring of the Grassmannian is generated by the Chern classes of 
$\cV_{i_{k+1}}/\cV_{i_{k-1}}$, by Proposition \ref{24}, we obtain \eqref{30a} because the Poincar\'e polynomial of a Grassmannian $\mathrm{Gr}(k,m)$ is $[m]_q!/[k]_q![m-k]_q!.$
\end{proof}

For example, when $I=\{r\}$ and $h\in\cH_{I,n}$ with $h(r)=n$, we find that $X_h$ is the Grassmannian $\mathrm{Gr}(r,n)$ and the equivariant Poincar\'e polynomial is 
\beq\label{83} \cF(h)=\frac{[n]_q!}{[r]_q![n-r]_q!}\eeq
with the characteristic $\mathsf{h}_n$ of the trivial representation of $S_n$ suppressed. 

\begin{corollary}\label{31} Let $h\in\cH_{I,n}$ be in the form of \eqref{27} and $\tilde{h}\in \cH_n$ be defined by \eqref{28}. Then
\beq\label{32}
\cF(\tilde{h})=\cF(h)\prod_{k=1}^{r+1}[i_k-i_{k-1}]_q!.\eeq 
\end{corollary}
\begin{proof} Applying \eqref{30} repeatedly to the forgetful morphism 
$$\rho_{\tilde{h},h}:X_{\tilde{h}}\lra X_h,$$
we obtain \eqref{32}.
\end{proof} 
\begin{example}
(1) If $I=\emptyset$ and $h:\{n\}\to \{n\}$ is the unique map in $\cH_{\emptyset, n}$, $X_{\tilde{h}}$ is the full flag variety $\Fl(n)$ and \eqref{32} tells us that $\cF(\tilde{h})$ is \eqref{23c}.

(2) If $I$ is in the form of \eqref{27} and $h(i_k)=n$ for all $k$, then $X_h$ is the partial flag variety $\Fl_I(n)$ and $X_{\tilde{h}}$ is the full flag variety $\Fl(n)$. Hence \eqref{23c} and \eqref{32} gives us 
\beq\label{89} 
\cF(h)=\frac{[n]_q!}{\prod_{k=1}^{r+1}[i_k-i_{k-1}]_q!}\eeq
with the characteristic $\mathsf{h}_n$ of the trivial representation of $S_n$ suppressed as usual.
\end{example}

We also need the pushforwards by inclusion when we deal with blowups.
For $h,h'\in \cH_{I,n}$, we write $h\leq h'$ if $h(i)\leq h'(i)$ for all $i\in I$.
\begin{proposition} \label{33} Let $h\le h'\in \cH_{I,n}$.
The pushforward
    \beq\label{34a} \iota_*:A^*(X_h)\longrightarrow A^{*+d}(X_{h'})\eeq
by the canonical inclusion $\iota:X_h\subset X_{h'}$ is $S_n$-equivariant, where $d$ denotes the codimension of $X_h$ in $X_{h'}$.
\end{proposition}
\begin{proof}
The map $\iota$ is $T$-equivariant as both $X_h$ and $X_{h'}$ lie in $\Fl_I(n)$ as $T$-invariant subvarieties. So we have the Cartesian diagram
\[\xymatrix{
X_h\ar[r]^\iota & X_{h'}\\
X_h^T\ar@{^(->}[u]^\jmath \ar[r]^{\iota^T} & X^T_{h'}\ar@{^(->}[u]_{\jmath'}
}\]
of inclusions. By \cite[Theorems 6.2 and 6.3]{Ful}, we have the commutative diagram
\[\xymatrix{
A^*_T(X_h) \ar[rr]^{\iota_*} \ar[d]_{c_d(N_{X_h/X_{h'}}|_{X_h^T})\circ\jmath^*} && A^{*+d}_T(X_{h'})\ar[d]^{{\jmath'}^*}\\
A^{*+d}_T(X^T_h)\ar[rr]^{\iota^T_*} && A^{*+d}_T(X^T_{h'})
}\]
where $N_{X_h/X_{h'}}$ denotes the normal bundle to $X_h$ in $X_{h'}$. 
As $\iota^T$ is $S_n$-equivariant, the bottom arrow is $S_n$-equivariant. Since ${\jmath'}^*$ is injective and $\jmath^*$, ${\jmath'}^*$, and $c_d(N_{X_h/X_{h'}}|_{X_h^T})$ are $S_n$-equivariant, the top arrow is $S_n$-equivariant. 
By \eqref{20}, \eqref{34a} is $S_n$-equivariant as well. 
\end{proof}

By Theorem \ref{4}, a generalized Hessenberg variety $X_h$ is \emph{reducible} if and only if $h(i)=i$ for some $i\in I$.  
The cohomology of reducible $X_h$ can be computed from irreducible cases by the following.  
\begin{proposition}\label{34}  Let $h\in \cH_{I,n}$ with $h(i)=i$ for some $i\in I$. 
Using the notation of Theorem \ref{4} (2), we have $$\cF(h)=\cF(h')\cF(h'').$$ 
\end{proposition}
\begin{proof}
It follows from Theorem \ref{4} (2) and the definition of the Frobenius character where 
the product is defined as 
$$\ch(V')\ch(V'')=\ch\,\mathrm{Ind}^{S_n}_{S_i\times S_{n-i}}(V'\otimes V'')$$
for an $S_i$-module $V'$ and an $S_{n-i}$-module $V''$. 
\end{proof}

\bigskip

\section{Birational geometry of generalized Hessenberg varieties}\label{36}

In this section, we show that certain natural morphisms between generalized Hessenberg varieties are actually blowups along smooth subvarieties which are also  generalized Hessenberg varieties. As a consequence, we obtain useful relations among the equivariant Poincar\'e polynomials of generalized Hessenberg varieties.
These relations provide us with an elementary proof of the Shareshian-Wachs conjecture \cite{SW, BC, GP2}. 
In \S\ref{71}, we present and prove a generalized Shareshian-Wachs conjecture that involves generalized Hessenberg varieties.
In \S\ref{91}, we see that the relations give us algorithms to compute the equivariant Poincar\'e polynomials of generalized Hessenberg varieties, which can be easily implemented on computer programs.

The following notation is useful. 
\begin{definition}\label{130}
Let $I\subset [n-1]$ be a set in the form of \eqref{27}. Let $1\le k\le r$ and $j=i_k$.  
\begin{enumerate}
    \item Let 
    $\kappa_{j}:\cH_{I,n}\to\cH_{I-\{j\},n}$ 
    be the function defined by 
   \[\kappa_jh(i) =\begin{cases}h(i) &\text{ if } i\notin h^{-1}(j) \\ 
   i_{k+1} & \text{ if }i\in h^{-1}(j).\end{cases}\]
    \item Let  
    $\tau_j: \{h\in \cH_{I,n}:h(i_{k-1})<h(i_k),~ j\le h(j)-1\in I\}\to\cH_{I,n}$ 
    be the function 
    defined by
    \[\tau_jh(i)= \begin{cases} h(i) &\text{ if }i\neq j\\
    h(j)-1 & \text{ if }i=j\end{cases}\]
    \item Let  
  $\tau^j:\{h\in \cH_{I,n}:h(i_k)<h(i_{k+1}), h(i_k)+1\in I\cup\{n\}\}\to \cH_{I,n}$
    be the function 
    defined by
    \[\tau^jh(i)= \begin{cases} h(i) & \text{ if }i\neq j\\
    h(j)+1 & \text{ if }i=j.\end{cases}\]
    We write $\tau_j^{j'}:=\tau^{j'}\tau_j$ for $j,j'$, and $\tau_{j,j'}:=\tau_{j'}\tau_{j}$ for $j< j'$.
\end{enumerate}
\end{definition}
For example, if we write $h=(h(1),\cdots,h(n))$, then for $h=(3,4,4,4)$, $\tau_1h=(2,4,4,4)$ and $\tau^1h=(4,4,4,4)$ while $\kappa_3h:\{1,2,4\}\to \{1,2,4\}$ is defined by $\kappa_3h(i)=4$ for all $i=1,2,4$.

\subsection{Deletion and blowup} \label{40}

In this section, we show that certain natural morphisms between generalized Hessenberg varieties are blowups under favorable circumstances. Recall that we write $\cH_{r,n}=\cH_{[r],n}$. 

For $h\in \cH_{r,n}$, there is a natural morphism
    \beq \label{41}
    \rho_j:X_h \longrightarrow X_{\kappa_jh},
    \eeq
    which deletes the subspace $V_j$ of dimension $j$. For $1\le j< r$, let $j_+=j+1$ and let $j_+=n$ for $j=r$ (or equivalently, let $r_+=n$). 
\begin{theorem} \label{42}
Let $h\in \cH_{r,n}$ and let $1\leq j\leq r$. Let $I=[r]-\{j\}$. The morphism $\rho_j$ in \eqref{41} sits in the following blowup diagram
\beq\label{124}\xymatrix{ E\ar[r] \ar[d] & X_h \ar[d]^{\rho_j}\\
Z\ar[r] & X_{\kappa_jh}}\eeq
if one of the following holds.
\begin{enumerate}
    \item $h(j)=h(j_+)$ and $h^{-1}(j)=\{j_0\}$: In this case, the blowup center is $Z=X_{\kappa_j\tau_{j_0}h}$
    of codimension $j_+-j+1$ and the exceptional divisor is $E=X_{\tau_{j_0}h}$. The equivariant Poincar\'e polynomials of these generalized Hessenberg varieties satisfy
\beq\label{51}
   \cF(h)=\cF(\kappa_jh)+q[j_+-j]_q\cF(\kappa_j\tau_{j_0}h). 
\eeq    
    \item $j<r$, $h(j+1)=h(j)+1> j+1$ and $h^{-1}(j)=\emptyset$: In this case, the blowup center is
    $Z=X_{\tau_{j+1}\kappa_j h}$ %\subset X_{\kappa_jh}\]
    of codimension 2 and the exceptional divisor is $E=X_{\tau_{j+1}h}$. The equivariant Poincar\'e polynomials of these generalized Hessenberg varieties satisfy
\beq\label{52}
   \cF(h)=\cF(\kappa_jh)+q\cF(\tau_{j+1}\kappa_jh). 
\eeq    
\end{enumerate}
\end{theorem}

\medskip
\begin{example}\label{121}
	Let $h:\{1,2,4\}\to \{1,2,4\}$ be a generalized Hessenberg function with $h(1)=2$ and $h(2)=h(4)=4$. 
	Then, 
	\[\rho_2:X_h\lra X_{\kappa_2h}=\PP^3, \quad (V_1\subset V_2\subset \C^4)\mapsto (V_1\subset \C^4)\]
	is bijective over $\PP^3-Z$, where
	\[Z:=X_{\kappa_2\tau_1h}=\{[1:0:0:0],[0:1:0:0],[0:0:1:0],[0:0:0:1]\},\]
	%\[\PP^3-X_{\kappa_2\tau_1h}~=~\PP^{3}-\{[1:0:0:0],[0:1:0:0],[0:0:1:0],[0:0:0:1]\},\] 
	since $V_2$ can be uniquely reconstructed by the equality $V_2=V_1+xV_1$. On the other hand, fibers over $Z$ are $\PP(\C^4/V_1)\cong\PP^2$, since $V_1=xV_1$. In fact, $\rho_2$ is the blowup of $\PP^3$ along $Z$. See the proof of Theorem~\ref{42} below.
	
	By the blowup formula~\eqref{53} below, 
	we have
	\beq\label{131}
	\cF(h)=\cF(\kappa_2h)+q[2]_q\cF(\kappa_2\tau_1h) =[4]_q\mathsf{h}_4+q[2]_q\mathsf{h_3}\mathsf{h_1}\eeq 
	where the second equality holds by \eqref{23b} and Proposition~\ref{34}.
	
	Using the notation right below Definition \ref{130}, upon multiplying $[3]_q!$, \eqref{131} gives us the equality  
	\[[3]_q\cF(2,4,4,4)=\cF(4,4,4,4)+q[2]_q\cF(1,4,4,4)\]
	for Hessenberg varieties 
	by \eqref{32} 
	because
	\[\rho_3:X_{(2,4,4,4)}\lra X_h\]
	is a $\PP^1$-bundle and 
	\[\rho_2\circ \rho_3:\quad X_{(4,4,4,4)}\cong \Fl(4)\lra \PP^3 \and X_{(1,4,4,4)}\lra Z\] 
	are $\Fl(3)$-bundles by Theorem~\ref{26}.
\end{example}

\medskip

\noindent\emph{Proof of Theorem~\ref{42}}.
%\begin{proof}
The assertions on the codimensions follow directly from Theorem \ref{4}.
By the Fujiki-Nakano criterion~\cite{FuNa}, it suffices to show the following to see that \eqref{124} is a blowup diagram.
\begin{enumerate}
		\item[(a)] $\rho_j$ is an isomorphism over $X_{\kappa_jh}- Z$,
		\item[(b)] $E\cong \PP(N_{Z/X_{\kappa_jh}})$ where $N_{Z/X_{\kappa_jh}}$ is the normal bundle to $Z$ in $X_{\kappa_jh}$,
		\item[(c)] $E$ is a Cartier divisor in $X_h$ with $\sO_{X_h}(E)|_E\cong \sO_{\PP(N_{Z/X_{\kappa_jh}})}(-1)$.
	\end{enumerate}
	
(1) Let $Z:=X_{\kappa_j\tau_{j_0}h}$, $E:=X_{\tau_{j_0}h}$,  $X:=X_{\kappa_jh}$ and $\tX:=X_h$. 
Under the assumptions of (1), $\rho_j$ is an isomorphism over $X-Z$, since $V_j$ can be uniquely reconstructed by the equality $V_j=V_{j-1}+xV_{j_0}$. Indeed, we have $\dim(V_{j-1}+xV_{j_0})=j$ on $X-Z$, since the conditions
\[(\kappa_jh)(j_0-1)=h(j_0-1)< j \and (\kappa_j\tau_{j_0}h)(j_0)=j-1\]
imply that 
$xV_{j_0}\cap V_{j-1}\supset xV_{j_0-1}$ on $X$, and $xV_{j_0}\subset V_{j-1}$ only on $Z$, so that 
\[V_{j-1}\cap xV_{j_0}=xV_{j_0-1} \quad \text{ on }X-Z.\]
On the other hand, $V_{j-1}+xV_{j_0}\subset V_j$ on $\tX$, since $h(j_0)=j$.

By Theorem \ref{26}, we have an isomorphism
$E=X_{\tau_{j_0}h}\cong \PP(\cV_{j_+}/\cV_{j-1})|_Z.$
Since $Z\subset X$ is the zero locus of the section of the vector bundle 
\beq\label{44} (\cV_{j_0}/\cV_{j_0-1})^*\otimes (\cV_{j_+}/\cV_{j-1})\eeq
of rank $j_+-j+1$ defined by 
$$x:\cV_{j_0}/\cV_{j_0-1}\lra \cV_{j_+}/\cV_{j-1},$$
$x$ is a transversal section as $Z$ is a smooth subvariety of codimension $j_+-j+1$ by Theorem \ref{4}. Hence the normal bundle $N_{Z/X}$ to $Z$ in $X$ is \eqref{44} restricted to $Z$. 
Since $(\cV_{j_0}/\cV_{j_0-1})^*$ is a line bundle, we have isomorphisms 
$$\PP N_{Z/X}\cong \PP(\cV_{j_+}/\cV_{j-1})|_Z\cong E,$$
$$\sO_{\PP N_{Z/X}}(-1)=(\cV_{j_0}/\cV_{j_0-1})^*\otimes\sO_{\PP (\cV_{j_+}/\cV_{j-1})|_Z}(-1).$$

By the same argument, the normal bundle $N_{E/\tX}$ to $E$ in $\tX$ is the line bundle
\beq\label{45} 
(\cV_{j_0}/\cV_{j_0-1})^*\otimes (\cV_{j}/\cV_{j-1})
=(\cV_{j_0}/\cV_{j_0-1})^*\otimes \sO_{\PP(\cV_{j_+}/\cV_{j-1})|_Z}(-1)
=\sO_{\PP N_{Z/X}}(-1).\eeq
This implies that $\rho_j:\tX\to X$ is the blowup along $Z$. 

By the blowup formula in \cite[\S6.7]{Ful} or \cite{GrHa}, 
we have an isomorphism
\beq\label{53}
A^*(\tX)\cong \rho_j^*A^*(X)\oplus \jmath_*A^*(E)/
\rho_j^*\imath_*A^*(Z)\eeq 
where $d+1$ is the codimension of $Z$ in $X$, and $\imath:Z\to X$ and $\jmath:E\to \tX$ denote the inclusion. Since $A^*(E)\cong A^*(Z)\otimes A^*(\PP^d)$ by \cite[Theorem 3.3]{Ful}, we find that  \eqref{53} is an isomorphism of $S_n$-representations by Propositions \ref{21},  \ref{33}, \eqref{44} and \eqref{45}. 
Upon applying the Frobenius characteristic, we obtain the formula \eqref{51}.

(2) Let $Z:=X_{\tau_{j+1}\kappa_j}h$, $X:=X_{\kappa_jh}$, $E:=X_{\tau_{j+1}h}$ and $\tX:=X_h$. 
The same argument as in (1) shows that $\rho_j$ is an isomorphism over $X-Z$, since $V_j$ can be uniquely reconstructed by the equality $V_j=V_{j+1}\cap x^{-1}V_{h(j+1)-1}$.

By Theorem \ref{26}, $E\cong \PP(\cV_{j+1}/\cV_{j-1})|_Z$. 
As in (1) above, the normal bundle $N_{Z/X}$ to $Z$ in $X$ is 
\[ (\cV_{j+1}/\cV_{j-1})^*\otimes (\cV_{h(j)+1}/\cV_{h(j)})|_Z\]
which is isomorphic to
\beq\label{46} (\cV_{j+1}/\cV_{j-1})\otimes \det (\cV_{j+1}/\cV_{j-1})^* \otimes (\cV_{h(j)+1}/\cV_{h(j)})|_Z\eeq
since $(\cV_{j+1}/\cV_{j-1})^*$ is a rank 2 bundle. Since $\det (\cV_{j+1}/\cV_{j-1})^* \otimes (\cV_{h(j)+1}/\cV_{h(j)})$ is a line bundle, we have an isomorphism
$$E\cong \PP(\cV_{j+1}/\cV_{j-1})|_Z\cong \PP N_{Z/X}.$$

The normal bundle $N_{E/\tX}$ to $E$ in $\tX$ is 
$$(\cV_{j+1}/\cV_{j})^*\otimes (\cV_{h(j)+1}/\cV_{h(j)})|_E$$
which is isomorphic to 
\beq\label{47}
\sO_{\PP(\cV_{j+1}/\cV_{j-1})}(-1)\otimes \det (\cV_{j+1}/\cV_{j-1})^* \otimes (\cV_{h(j)+1}/\cV_{h(j)})|_E\eeq
since $\O_{\PP(\cV_{j+1}/\cV_{j-1})}(-1)=\cV_j/\cV_{j-1}$ and 
$$\cV_j/\cV_{j-1}\otimes \cV_{j+1}/\cV_{j}\cong \det (\cV_{j+1}/\cV_{j-1}).$$
Comparing \eqref{46} and \eqref{47}, we find that
$$N_{E/\tX}\cong \sO_{\PP N_{Z/X}}(-1).$$
This proves that $\rho_j$ is the blowup along $Z$ with exceptional divisor $E$. 

The proof of \eqref{52} is the same as that of \eqref{51} and we omit it. 
%\end{proof}
\hfill$\square$

\bigskip

\begin{corollary}[Case $j=r$]\label{48}
    Let $h\in \cH_{r,n}$. Suppose that $h(r)=n$.
    \begin{enumerate}
        \item If $h^{-1}(r)=\emptyset$, then $\rho_r$ is the projectivized vector bundle $$\PP(\sO^{\oplus n}_{X_{\kappa_rh}}/\cV_{r-1})\lra X_{\kappa_rh},$$
and we have            
\beq\label{73a} \cF(h)=[n-r+1]_q \cF(h|_{[r-1]\cup\{n\}}).\eeq
        \item If $h^{-1}(r)=\{j\}$, then $\rho_r$ is the blowup along the subvariety $X_{\kappa_r\tau_{j}h}$
        of codimension $n-r+1$, and we have 
\beq\label{74}\cF(h)=\cF(((\tau^j)^{n-r}h)|_{[r-1]\cup\{n\}})+q[n-r]_q \cF((\tau_jh)|_{[r-1]\cup\{n\}}).\eeq 
        \item If $h$ is strictly increasing on $\{i~:~ h(i)<n\}$ and irreducible, then the forgetful morphism 
 \beq\label{49} \rho:X_h\lra \PP^{n-1}, \quad (V_i)_{i\in [r]}\mapsto V_1\eeq
        is the composition $\rho_2\circ\cdots \circ \rho_r$ where each $\rho_i$ is either a projective bundle or a blowup along a generalized Hessenberg variety $X_{h'}$ for some $h'\in \cH_{r',n}$ with $r'<r$. 
    \end{enumerate}
\end{corollary}
\begin{proof}
(1) follows from Theorem \ref{26} 
and the fact that $\kappa_rh$ is equal to the restriction $h|_{[r-1]\cup\{n\}}$ of $h$ to $[r-1]\cup\{n\}$. 
(2) follows from Theorem \ref{42} and the equalities $\kappa_rh=(\tau^j)^{n-r}h|_{[r-1]\cup\{n\}}$ and $\kappa_r\tau_jh=(\tau_jh)|_{[r-1]\cup\{n\}}$.

If $h$ is strictly increasing on $[r]-h^{-1}(n)$, either (1) or (2) applies for $\rho_r$ and $\kappa_rh$ is also strictly increasing on $[r-1]-(\kappa_rh)^{-1}(n)$. The formulas \eqref{73a} and \eqref{74} follow from \eqref{53}. 
\end{proof}    

\smallskip
\begin{example}\label{50} Let $h_1:[n]\to [n]$, $h_1(i)=\min\{i+1,n\}$. The forgetful morphism \eqref{49} factors as 
\[\rho_2\circ \cdots \circ \rho_{n-1}:X_{h_1}=X_{h^{(n-1)}}\lra X_{h^{(n-2)}}\lra \cdots 
\longrightarrow X_{h^{(1)}}=\PP^{n-1}\]
into the blowups $\rho_{j+1}:X_{h^{(j+1)}}\xrightarrow{} X_{h^{(j)}}$
for $1\leq j \leq n-2$, where $h^{(j)}=\kappa_{j+1}h^{(j+1)}\in \cH_{j,n}$ is defined by 
$$h^{(j)}(i)=\left\{ \begin{matrix} i+1 & \text{ for} &1\le i<j\\
n & \text{ for} &i=j.\end{matrix}\right.$$
The blowup center of $\rho_{j+1}$ is the disjoint union of $\binom{n}{j}$ copies of the Hessenberg variety corresponding to the Hessenberg function $[j]\to [j]$ defined by $i\mapsto \min\{i+1,j\}$.

The above blowup sequence is exactly the blowup construction of the permutohedral variety \cite{Pro,Kap} and hence $X_{h_1}$ is isomorphic to the permutohedral variety.

Let $Q_n:=\cF(h_1)$ be the equivariant Poincar\'e polynomial of $X_{h_1}$ for $h_1\in\cH_n$. Set $Q_1:=1$. Applying \eqref{51} to $\rho_{j+1}$ above, we have
\beq\label{54}
    \begin{split}
        \cF(h^{(j+1)})&=\cF(h^{(j)})+q[n-j-1]_q \cF(\kappa_{j+1}\tau_jh)\\
        &=\cF(h^{(j)})+q[n-j-1]_q\mathsf{h}_{n-j}Q_j,
    \end{split}
\eeq 
where $\mathsf{h}_{n-j}$ denotes the Frobenius characteristic of the trivial representation of $S_{n-j}$. 
From \eqref{54}, we obtain the recursive formula for $Q_n$: 
\beq\label{55}
Q_n=[n]_q+q\sum_{j=1}^{n-2}[n-j-1]_q \mathsf{h}_{n-j}Q_j.\eeq
\end{example}

\begin{remark}
The Hessenberg variety $X_{h_1}$ in Example \ref{50} is also isomorphic to the Losev-Manin space $\overline{LM}_n$ in \cite{LM} of $n$-pointed chains of projective lines, which is just one of many compactifications of the moduli space ${M}_{0,n+2}$ of $n+2$ points in $\PP^1$ up to the automorphism group action of $\PP^1$.  
In \cite{Has}, Hassett constructed a moduli theoretic compactification $\overline{M}_A$ of $M_{0,n+2}$ for any sequence $A=(a_1,a_2,\cdots, a_{n+2})$ of rational numbers with $0< a_i\le 1$, $\sum_i a_i> 2$. The Losev-Manin space $\overline{LM}_n$ is actually the Hassett moduli space $\overline{M}_{(1,1,\epsilon,\cdots,\epsilon)}$ for $\epsilon>0$ sufficiently small. The blowups in Example \ref{50} are actually 
the \emph{wall crossings} of the moduli spaces $$\overline{M}_{(1,\delta,\epsilon,\cdots,\epsilon)}$$
as $\delta$ varies from $1$ to $1-n\epsilon$. 

Techniques in this paper apply to the study of the cohomology of various compactifications of $M_{0,n}$. 
See, for instance, \cite{CKL} for the study of $S_n$-representations on the cohomology of $\overline{M}_{0,n}=\overline{M}_{(1,1,\cdots,1)}$. 
\end{remark}

\subsection{An elementary proof of the Shareshian-Wachs conjecture}\label{60} 
In this subsection, we prove 
 the Shareshian-Wachs conjecture in \cite{SW}
 as an easy application of Theorem \ref{42}. 
Our proof is elementary in the sense that we only use classical algebraic geometry in  \cite{Hart, Ful, GrHa} and the arguments are straightforward. 

The Shareshian-Wachs conjecture relates the equivariant Poincar\'e polynomial $\cF(h)$ in \eqref{23a} of a Hessenberg variety $X_h$ with the chromatic quasi-symmetric function $\csf_q(h)$ of the graph $\Gamma_h$ associated to $h$.

\begin{definition}\label{61}
Let $h:[n]\to [n]$ be a Hessenberg function. The \emph{graph} $\Gamma_h$ \emph{associated to} $h$ consists of the set of vertices $\V(\Gamma_h)=[n]$ and the set of edges 
$$\E(\Gamma_h)=\{(i,j)\,:\, 1\le i<j\le h(i)\}.$$
A proper coloring of $\Gamma_h$ is a map $\gamma:\V(\Gamma_h)\to \N$ such that $\gamma(i)\ne \gamma(j)$ whenever $(i,j)\in \E(\Gamma_h)$. 
The \emph{chromatic quasi-symmetric function} of $\Gamma_h$ is defined as 
   \[\csf_q(h)=\sum_{\substack{\gamma:\V(\G_h)\to\N\\\text{proper}}}q^{\asc_{h}(\gamma)}x_{\gamma} \in \Lambda^n[q]\]
    where $x_{\gamma}:=\prod_{i=1}^n x_{\gamma(i)}$ and 
    $$\asc_h(\gamma)=|\{(i,j)\in \E(\G_h)\,:\, \gamma(i)<\gamma(j)\}|.$$
Here  $\Lambda^n$ denotes the subgroup of $\varprojlim_m \Z[x_1,\cdots,x_m]^{S_m}$ which consists of homogeneous symmetric functions of degree $n$ and $|\cdot |$ denotes the number of elements. 
\end{definition}

The celebrated \emph{Stanley-Stembridge conjecture} in \cite{SS} asks if all the coefficients of the chromatic quasi-symmetric function $\csf_q(h)$ are positive as a polynomial in $q$ and the elementary symmetric functions. 

Let $\omega:\Lambda \to\Lambda $ denote the involution of the ring $\Lambda=\bigoplus_{m}\Lambda^m$ sending a Schur function to its transpose.
This is a graded ring isomorphism. We also denote by $\omega:\Lambda[q] \xrightarrow{}\Lambda[q]$ its natural extension to $\Lambda[q]$, which sends $aq^i$ to $\omega(a)q^i$ for $a\in \Lambda$ and $i\geq0$. 

\begin{theorem}[Shareshian-Wachs conjecture]\label{62} \cite{SW, BC, GP2}
$\omega\cF(h)=\csf_q(h).$
\end{theorem}
This is a beautiful theorem relating a purely combinatorial object $\G_h$ with the geometric object $X_h$. In \cite{BC}, Brosnan and Chow proved Theorem \ref{62} by studying degeneration and monodromy. In \cite{GP2}, Guay-Paquet proved this by using the Hopf algebra on Dyck paths. 

Our proof of Theorem \ref{62} is based on the following.
\begin{theorem}\cite[Theorem 1.1]{AN}\label{63}
$\csf_q(h)$ is the unique function satisfying 
\begin{enumerate}
\item the modular law in Proposition \ref{64} below,
\item $\csf_q(h)=\csf_q(h') \csf_q(h'')$ in the notation of Theorem \ref{4} (2), if $h(i)=i$ for some $i\in [n]$, and
\item $\csf_q(h)=[n]_q!\mathsf{e}_n$ if $h(i)=n$ for all $i\in [n]$. 
\end{enumerate}
\end{theorem}

\noindent{\em Proof of Theorem~\ref{62}.}
%\begin{proof}
Obviously Theorem \ref{62} is proved if we show that $\omega \cF(h)$ satisfies the three conditions in Theorem \ref{63}. By \eqref{23c}, we have $\omega \cF(h)=[n]_q! \mathsf{e}_n$ since $\omega \mathsf{h}_n=\mathsf{e}_n$. Also Proposition \ref{34} establishes (2) in Theorem \ref{63} for $\cF(h)$ and hence the same holds for $\omega \cF(h)$ because $\omega$ is a ring homomorphism. Therefore Theorem \ref{62} follows from the following proposition.
%\end{proof}
\hfill $\square$

\begin{proposition}[The modular law] \label{64} 
Let $h\in \cH_n$ and let $1\leq j<n$. \begin{enumerate}
    \item If $h(j)=h(j+1)$ and $h^{-1}(j)=\{j_0\}$, then
    \beq\label{69} [2]_q\cF(h)=\cF(\tau^{j_0}h)+q \cF(\tau_{j_0}h).\eeq
    \item If $h(j+1)=h(j)+1> j+1$ and $h^{-1}(j)=\emptyset$, then
    \beq\label{70} [2]_q\cF(h)=\cF(\tau^{j}h)+q \cF(\tau_{j+1}h).\eeq
\end{enumerate}
\end{proposition}
\begin{proof}
(1) Notice that $\kappa_jh=\kappa_j\tau^{j_0}h$ under our assumption. By Theorem \ref{42} (1), we have
\beq\label{65} \cF(h)=\cF(\kappa_jh)+q\cF(\kappa_j\tau_{j_0}h),\quad \cF(\tau_{j_0}h)=[2]_q\cF(\kappa_j\tau_{j_0}h).\eeq
By Theorem \ref{26} and \eqref{30}, we have 
\beq\label{66} \cF(\tau^{j_0}h)=[2]_q\cF(\kappa_j\tau^{j_0}h)=[2]_q\cF(\kappa_jh).\eeq
It is easy to deduce \eqref{69} from \eqref{65} and \eqref{66}.

(2) By Theorem \ref{42} (2), we have 
\beq\label{67} \cF(h)=\cF(\kappa_jh)+q\cF(\kappa_j\tau_{j+1}h),\quad
\cF(\tau_{j+1}h)=[2]_q\cF(\kappa_j\tau_{j+1}h).\eeq
By Theorem \ref{26}, we have
\beq\label{68} \cF(\tau^jh)=[2]_q\cF(\kappa_j\tau^jh)=[2]_q\cF(\kappa_jh).\eeq
It is easy to deduce \eqref{70} from \eqref{67} and \eqref{68}.
\end{proof}

\begin{remark}
Theorem \ref{62} can also be proved by our Algorithm 1 in \S\ref{103}, instead of Theorem \ref{63}.  
By Proposition \ref{73}, we have the formulas \eqref{75} and \eqref{76}. It is straightforward to see that the same formulas hold for $\csf_q(h)$ by the modular law.
By applying Algorithm 1 to $\csf_q(h)$, Theorem \ref{63} is reduced to the statement about $h_1$ (Example \ref{50}) where the theorem is obvious. 

\end{remark}

\begin{remark} Our motivation for this paper is to apply the techniques developed in \cite{CKL} to Hessenberg varieties. 
While we were completing this paper, we learned that Proposition \ref{64} was proved by Precup and Sommers in \cite{PS} by a different method and the above proof of (1) was sketched in \cite[Example 3.5]{AN2} without details.  
See \cite{KL,HMS} for an analogous modular law for the twin manifolds of Hessenberg varieties \cite{AB}.
\end{remark}

In \S\ref{71} below, we present natural generalizations of the chromatic quasi-symmetric functions, the Stanley-Stembridge conjecture and the Shareshian-Wachs conjecture, adapted for generalized Hessenberg functions. Then we prove the generalized Shareshian-Wachs conjecture.

\subsection{Modifications}\label{72}

In this subsection, we use Theorems \ref{26} and \ref{42} to find more relations among generalized Hessenberg varieties. These are used for new algorithms to compute the equivariant Poincar\'e polynomials of generalized Hessenberg varieties in \S\ref{91}. 

By Theorem \ref{42}, we have the following rational maps, which we call \emph{modifications}.
\begin{proposition}\label{73}
    Let $h\in \cH_{r,n}$. Suppose that there exists $1\leq j<r$ such that $h(j-1)<h(j)=h(j+1)\neq j+1$.
    \begin{enumerate}
        \item If $h^{-1}(j)=\emptyset$, then the rational map
        $X_h\dashrightarrow X_{\tau_jh}$
        which generically replaces $V_j$ by $V_{j+1}\cap x^{-1}V_{h(j+1)-1}$ can be resolved as 
        \[\xymatrix{ 
        X_h \ar@{-->}[rr]\ar[rd]_{\rho_j} && X_{\tau_{j}h}\ar[ld]^{\rho_j}\\
        &X_{\kappa_jh}=X_{\kappa_j\tau_jh}&}\]
        where
        \begin{enumerate}
            \item the first $\rho_j$ is a $\PP^1$-bundle,
            \item the second $\rho_j$ is the blowup along $X_{\tau_{j+1}\kappa_jh}$.
        \end{enumerate}
        In this case, we have 
\beq\label{75} \cF(h)=[2]_q\cF(\tau_jh)-q\cF(\tau_{j,j+1}h).\eeq
        \item If $h^{-1}(j)=\{j_0\}$, then the rational map
        $X_h\dashrightarrow X_{\tau^{j_0}_jh}$
        which generically replaces $V_j$ by $V_{j+1}\cap x^{-1}V_{h(j+1)-1}$ can be resolved as 
        \[\xymatrix{ 
        X_h \ar@{-->}[rr]\ar[rd]_{\rho_j} && X_{\tau^{j_0}_{j}h}\ar[ld]^{\rho_j}\\
        &X_{\kappa_jh}=X_{\kappa_j\tau^{j_0}_jh}&}\]
        where 
        \begin{enumerate}
            \item the first $\rho_j$ is the blowup along $X_{\kappa_j\tau_{j_0}h}$,
            \item the second $\rho_j$ is the blowup along $X_{\tau_{j+1}\kappa_jh}$.
        \end{enumerate}
    \end{enumerate}
    In this case, we have
\beq\label{76} \cF(h)=\cF(\tau^{j_0}_jh)-q\cF(\tau_{j,j+1}h)+q\cF(\tau_{j_0,j}h).\eeq
\end{proposition}
\begin{proof}
The diagrams in (a) and (b) are immediate from Theorems \ref{26} and \ref{42}. 
In (1), the formula \eqref{75} follows from \eqref{30} and \eqref{52} since $X_{\tau_{j,j+1}h}$ is a $\PP^1$-bundle over $X_{\tau_{j+1}\kappa_jh}$ by Theorem \ref{30}. 

In (2),  from \eqref{51} and from \eqref{52},  we have 
\beq\label{76a}\cF(h)-q\cF(\kappa_j\tau_{j_0}h) =\cF(\kappa_jh)=\cF(\tau^{j_0}_jh)-q\cF(\tau_{j+1}\kappa_jh).\eeq
Applying \eqref{51} with $h$ replaced by $\tau_{j,j+1}h$, we have 
\beq\label{76b}\cF(\tau_{j,j+1}h)=\cF(\tau_{j+1}\kappa_jh)+q\cF(\tau_{j+1}\kappa_j\tau_{j_0}h).\eeq
Applying \eqref{52} with $h$ replaced by $\tau_{j_0,j}h$, we have
\beq\label{76c}\cF(\tau_{j_0,j}h)=\cF(\kappa_j\tau_{j_0}h)+q\cF(\tau_{j+1}\kappa_j\tau_{j_0}h).\eeq
The formula \eqref{76} follows from \eqref{76a}, \eqref{76b} and \eqref{76c}.
\end{proof}
It is easy to see that the modular law \eqref{70} is equivalent to \eqref{75} by Corollary \ref{31}. 

We say that $h$ is \emph{modified} to $\tau_jh$ or $\tau_{j}^{j_0}h$ by the modifications at $j$ in Proposition \ref{73} (1) or (2).

One immediate advantage of these modifications is that they complement Corollary \ref{48} which 
reduces $X_h$ for $h\in \cH_{r,n}$ to $X_{\kappa_rh}$ with $\kappa_rh\in \cH_{r-1,n}$, when 
$\lvert h^{-1}(r)\rvert\le 1$. When $\lvert h^{-1}(r)\rvert\geq 2$, we can apply Proposition \ref{73} to modify $X_h$ to $X_{\hat{h}}$ with $| \hat{h}^{-1}(r) |\le 1.$ 

\begin{corollary} \label{77}
Let $h\in \cH_{r,n}$. Suppose that $h(r)=n$ and $\lvert h^{-1}(r)\rvert \geq2$. 
Let $\{r_i\}$ denote the sequence of integers defined recursively by $$r_0=r, \ \ r_i=\min h^{-1}(r_{i-1}) \ \ \text{ for }h^{-1}(r_{i-1})\ne \emptyset.$$ Let 
$j=\min\{i\,:\, |h^{-1}(r_{i})|\le 1\}.$ Then we have a modification 
$$X_h\dashrightarrow X_{h'}, \quad h'=\tau_{r_j}h \text{ or } \tau_{r_j}^{r_{j+1}}h$$
by Proposition \ref{73}. By repeating this process finitely many times, we obtain a dominant rational map 
$$X_h\dashrightarrow X_{\hat{h}}\ \ \text{with }\hat{h}(r)=n \text{ and } |\hat{h}^{-1}(r)|\le 1$$
which is a composition of the modifications in Proposition \ref{73}.
\end{corollary}
\begin{proof}
Such a $j$ exists because $[n]\supset h^{-1}(r)\sqcup h^{-1}(r_1)\sqcup h^{-1}(r_2)\sqcup \cdots$ is a finite set. 
The modification process terminates in finite time, since each modification drops the sum $\sum_{i=1}^rh(i)$ by one or strictly increases the sequence $(h(1),\cdots,h(r))$ in the lexicographic order among those sequences having the same sum $\sum_{i=1}^r h(i)$.
\end{proof}

\begin{definition}\label{84}
    For $k\geq 1$, let $h_k:[n]\xrightarrow{}[n]$ denote the Hessenberg function defined by 
    \[h_k(i)=\min\{i+k,n\}.\]
    \end{definition}

When $h\geq h_k$, we can always modify $h$ to $h_k$. 
\begin{corollary} \label{78}
    Let $k\geq1$. Let $h\in \cH_n$ be a Hessenberg function satisfying $h\geq h_k$. Then, there exists a dominant rational map $X_h \dashrightarrow X_{h_k}$ which is a composition of the modifications in Proposition \ref{73}.
\end{corollary}
\begin{proof}
Note that $h=h_k$ if and only if $h\geq h_k$ and
\[\{1\leq i< n-k~:~h(i)=h(i+1)\}=\emptyset.\]
If $h\ne h_k$, there exists $1\le j<n-k$ with $h(j)=h(j+1)$. 
For the smallest such $j$, we have $|h^{-1}(j)|\le 1$
and hence we can apply Proposition \ref{73} to modify $h$ to $\tau_jh$ or $\tau^{j_0}_jh$. 
Note that if $h>h_k$, then the resulting $\tau_jh$ or $\tau^{j_0}_jh$ is greater than or equal to $h_k$.  
So we can repeatedly apply this modification until we reach $h_k$. 
This modification process stops in finite time 
because each modification drops the sum $\sum_{i=1}^n h(i)$ by one or strictly increases the sequence $(h(1),\cdots, h(n))$ in the lexicographic order among those sequences having the same sum $\sum h(i)$.
\end{proof}

In particular, for every $h\in \cH_n$ with $h(i)>i$ for all $i<n$, there is a dominant rational map $X_h\dashrightarrow X_{h_1}$ which is a composition of the modifications in Proposition \ref{73}. 
Since $\cF(h_1)=Q_n$ is given by the recursive formula \eqref{55}, we can compute $\cF(h)$ by keeping track of the modifications. 
In \S\ref{91}, we give explicit algorithms that calculate $\cF(h)$.

\bigskip

\section{The generalized Shareshian-Wachs conjecture}\label{71}

In \S\ref{60}, we provided an elementary proof of the Shareshian-Wachs conjecture (Theorem \ref{62}) which relates the chromatic quasi-symmetric function of the associated graph $\Gamma_h$ with the equivariant Poincar\'e polynomial of the Hessenberg variety $X_h$ when $h\in \cH_n$ is an ordinary Hessenberg function. 
In this section, we present a natural generalization of the Shareshian-Wachs conjecture for an arbitrary generalized Hessenberg function $h$ and prove it (Theorem \ref{86}).

\medskip

For a graph $\Gamma$, we have the \emph{chromatic function} 
\[\chi_\Gamma:\N\lra \Z, \quad \chi_\Gamma(n)=|\{\gamma:V(\Gamma)\to [n]\,:\, \textrm{proper}\}|\]
where a coloring $\gamma:V(\Gamma)\to [n]$ of the vertices is called proper if $\gamma(i)\ne \gamma(j)$ whenever there is an edge whose end points are $i$ and $j$. The chromatic function was introduced in the early 20th century by G. Birkhoff and H. Whitney to study the four color problem which says $\chi_\Gamma(4)>0$ for planar graphs.  
As we have seen in Definition \ref{61}, the chromatic function $\chi_\Gamma$ can be refined to the chromatic quasi-symmetric function $\csf_q$ for the graph associated to an ordinary Hessenberg function \cite{Sta,SW}. 
To generalize the Shareshian-Wachs conjecture for generalized Hessenberg functions, 
we extend these notions to weighted graphs as follows. 

\begin{definition} A graph $\Gamma$ is called \emph{weighted} if it comes with a map,  called the weight, 
$$w:\V(\Gamma)\lra \N$$
on the set of vertices. 
For a subset $S\subset \N$, an \emph{$S$-coloring} on the weighted graph $(\G,w)$ is a function 
    \[\gamma:\V(\G)\longrightarrow 2^S\]
    which associates to each vertex $i\in \V(\G)$ a subset $\gamma(i)\subset S$ of order $w(i)$. We simply call $\gamma$ a \emph{coloring} if $S=\N$.
 An $S$-coloring $\gamma$ on $(\G,w)$ is called \emph{proper} if $\gamma(i)$ and $\gamma(j)$ are disjoint whenever two vertices $i,j$ are the end points of an edge.
\end{definition}

The following is a natural generalization of $\chi_\Gamma$.
\begin{definition} The \emph{chromatic function} $\chi_{\G,w}$ of the weighted graph $(\G,w)$ is defined by
    \[\chi_{\G,w}(n)=|\{\text{proper $[n]$-colorings of }(\G,w)\}|\]
    for $n\in \N$.
\end{definition}
Obviously, $\chi_{\Gamma,w}$ equals $\chi_\Gamma$ if $w(i)=1$ for all $i\in \V(\G).$

\medskip 

For generalized Hessenberg functions, Definition \ref{61} is generalized as follows.
\begin{definition}\label{82} 
Let $h\in \cH_{I,n}$ be a generalized Hessenberg function with $I=\{i_1,\cdots,i_r\}$ and $0=i_0<i_1<\cdots<i_r<i_{r+1}=n$.
\begin{enumerate}
    \item The weighted graph $(\G_h,w_h)$ \emph{associated to $h$} consists of the vertex set $\V(\G_h)=I\cup \{n\},$ the edge set 
        \[\E(\G_h)=\{(i,j): i <j\leq h(i), ~ i,j\in \V(\Gamma_h)\}\]
and the weight $w_h$ defined by
        $w_h(i_k)=i_k-i_{k-1}.$
    In particular, the total weight $\lvert w_h\rvert=\sum_{i\in \V(\G)}w_h(i)$ is $n$.
    \item The \emph{chromatic quasi-symmetric function} $\csf_q(h)$ of the weighted graph $(\G_h,w_h)$ is defined by
    \[\csf_q(h)=\sum_{\substack{\gamma:\V(\G_h)\to2^{\N}\\\text{proper}}}q^{\asc_{h}(\gamma)}x_{\gamma} \in \Lambda^n[q]\]
    where $x_{\gamma}=\prod_{i\in \V(\G_h)}\prod_{a\in \gamma(i)}x_a$ and 
    \[\asc_{h}(\gamma)=\sum_{\substack{(i,j)\in \E(\G_h)}} |\{(a,b)\in \gamma(i)\times \gamma(j):a<b\}|.\]
\end{enumerate}
\end{definition}
It is obvious that when $h\in \cH_n$ is an ordinary Hessenberg function, we get back to Definition \ref{61}.
\begin{remark}
(1) After completing the first draft of this paper, we were informed by Jaeseong Oh that Definition \ref{82} was introduced independently in \cite{Hwa}. (See \cite{Gas} for the $q=1$ case.) It is proved in \cite{Hwa} that $\csf_q(h)$ for any generalized Hessenberg function $h$ is symmetric.
Actually it also follows from Lemma \ref{80} below.
Indeed, using the notation in the proof of Theorem \ref{86}, by  \eqref{80d},
$$\csf_q(h)=\csf_q(\tilde{h})/\prod_{k=1}^{r+1}[i_k-i_{k-1}]_q!$$
belongs to $\Lambda^n[q]$, since we know that $\csf_q(\tilde{h})\in \Lambda^n[q]$ for the ordinary Hessenberg function $\tilde{h}$. 

(2) It is obvious that when $h\in \cH_n$ is an ordinary Hessenberg function, $w_h(i)=1$ for all $i$ and Definition \ref{82} becomes  Definition \ref{61}.
\end{remark}

\begin{example}\label{90b} 
Let $h\in\cH_{I,n}$ be a generalized Hessenberg function.

(1) When $I=\emptyset$ and $h:\{n\}\to \{n\}$, the associated weighted graph consists of a single point with weight $n$ and we have 
\beq\label{90a}\csf_q(h)=\mathsf{e}_n=\omega\cF(h)\eeq 
where $\mathsf{e}_n=\sum_{i_1<\cdots<i_n}x_{i_1}\cdots x_{i_n}$ is the $n$-th elementary symmetric function, 
because $X_h$ is a point and $\omega \mathsf{h}_n=\mathsf{e}_n$ where $\mathsf{h}_n$ is the characteristic of the trivial representation of $S_n$. Here $\omega$ is the involution of $\Lambda[q]$ 
defined right before Theorem \ref{62} in \S\ref{60}.

(2) When $I=\{1\}$ and $h(1)=n$, the associated graph $\G_h$ is the path 
\[\mathbf{1}~\text{---}~\mathbf{n}\]
with two vertices $\mathbf{1}, \mathbf{n}$ and one edge connecting them. The weight function is defined by 
\[w_h(\mathbf{1})=1 ~\text{ and }~ w_h(\mathbf{n})=n-1.\] 
So, for every proper coloring $\gamma$, $\gamma(\mathbf{1})\cap \gamma(\mathbf{n})=\emptyset$, $|\gamma(\mathbf{1})|=1$ and $|\gamma(\mathbf{n})|=n-1$. For $S\subset\N$ with $|S|=n$ and a nonnegative integer $d$, the number of proper colorings $\gamma$ satisfying
\[\gamma(\mathbf{1})\sqcup \gamma(\mathbf{n})=S ~\text{ and }~ \asc_h(\gamma)=d\]
is equal to the number of words of length $n$ with one copy of $\mathbf{1}$ and $n-1$ copies of $\mathbf{n}$ having precisely $d$ pairs $(\mathbf{1},\mathbf{n})$ in which $\mathbf{1}$ precedes $\mathbf{n}$. This number is $1$ for $0\leq d\leq n-1$ and $0$ otherwise. Since $X_h=\PP^{n-1}$, we have 
\[\csf_q(h)=(1+q+\cdots + q^{n-1})\mathsf{e}_n=[n]_q\mathsf{e}_n=\omega\cF(h).\]

By the same argument, we find that when $I=\{r\}$ and $h(r)=n$,
$$\csf_q(h)=\frac{[n]_q!}{[r]_q![n-r]_q!} \mathsf{e}_n=\omega\cF(h)$$
by \eqref{83}.

(3) When $I=\{1,2\}$ and $h(1)=h(2)=n$, the associated graph $\G_h$ is the triangle with vertices $\mathbf{1},\mathbf{2}, \mathbf{n}$. The weight function is defined by 
\[w_h(\mathbf{1})=w_h(\mathbf{2})=1~\text{ and }~ w_h(\mathbf{n})=n-2.\] 
So, for every proper coloring $\gamma$, $\gamma(\mathbf{1}),\gamma(\mathbf{2}),\gamma(\mathbf{n})$ are mutually disjoint with $|\gamma(\mathbf{1})|=1$, $\gamma(\mathbf{2})|=1$ and $\gamma(\mathbf{n})=n-2$. 
For $S\subset\N$ with $|S|=n$ and a nonnegative integer $d$, the number of proper colorings $\gamma$ satisfying
\[\gamma(\mathbf{1})\sqcup \gamma(\mathbf{2})\sqcup \gamma(\mathbf{n})=S ~\text{ and }~ \asc_h(\gamma)=d\]
is equal to the number of pairs $(a,b)$ of integers, indicating the locations of $\mathbf{1}$ and $\mathbf{2}$ in a word of length $n$, satisfying
\[a+b=d \text{ or }d-1,\quad 0\leq a,b\leq n-2.\]
It is straightforward to see that the number equals the coefficient of $q^d$ in $[n]_q[n-1]_q$. Therefore, we have 
\[\csf_q(h)=[n]_q[n-1]_q\mathsf{e}_n=\omega \cF(h)\]
by \eqref{89} since $X_h$ is the partial flag variety $\Fl_{\{1,2\}}(n)$.

By the same argument, we find that when $I$ is in the form of \eqref{27} and $h(i_k)=n$ for all $k$,
\[\csf_q(h)=\frac{[n]_q!}{[i_1]_q![i_2-i_1]_q!\cdots [n-i_r]_q!}\mathsf{e}_n=\omega\cF(h).\]
When $I=[n-1]$ and $h(i)=n$ for all $i$, the graph $\G_h$ is the complete graph and 
\beq\label{90}\csf_q(h)=[n]_q!=\omega\cF(h)\eeq
since $X_h=\Fl(n)$ is the full flag variety. 

(4) When $I=\{1,2\}$, $h(1)=2$ and $h(2)=n$, the associated graph $\G_h$ is the path 
$\mathbf{1}~\text{---}~\mathbf{2}~\text{---}~\mathbf{n}$ 
and the weight is  
\[w_h(\mathbf{1})=w_h(\mathbf{2})=1~\text{ and }~w_h(\mathbf{n})=n-2.\]
In this case, for each proper coloring $\gamma$, either
\begin{enumerate}
    \item[(i)] $\gamma(\mathbf{1}),\gamma(\mathbf{2}),\gamma(\mathbf{n})$ are mutually disjoint, or
    \item[(ii)] $\gamma(\mathbf{1})\subset \gamma(\mathbf{n})$ and $\gamma(\mathbf{2})\cap \gamma(\mathbf{n})=\emptyset$.
\end{enumerate}
The contribution of the proper colorings of type (i) to $\csf_q(h)$ is
\[\sum_{k=0}^{n-2}\Big((n-k-1)q+(k+1)\Big)q^k\mathsf{e}_n=\Big([n]_q+nq[n-2]_q\Big)\mathsf{e}_n\]
while the contribution of type (ii) is 
\[\sum_{k=0}^{n-2}\Big(\sum_{j=k+2}^{n-1}\sum_{i_1<\cdots<i_{n-1}}x_{i_j}x_{i_1}\cdots x_{i_{n-1}}q+\sum_{j=1}^k\sum_{i_1<\cdots<i_{n-1}}x_{i_j}x_{i_1}\cdots x_{i_{n-1}}\Big)q^k\]
\[=q[n-2]_q\sum_{j=1}^{n-1}\sum_{i_1<\cdots<i_{n-1}}x_{i_j}x_{i_1}\cdots x_{i_{n-1}}=q[n-2]_q\Big(\mathsf{e}_{(n-1,1)}-n\mathsf{e}_n\Big)\]
where $\mathsf{e}_{(n-1,1)}=\mathsf{e}_{n-1}\mathsf{e}_1$.
Therefore, we have 
\[\csf_q(h)=[n]_q\mathsf{e}_n+q[n-2]_q\mathsf{e}_{(n-1,1)}=\omega\cF(h)\]
since $X_h$ is the blowup of $\PP^{n-1}$ 
along the $n$ coordinate points on which $S_n$ acts by permuting the coordinates.
\end{example}

The above examples indicate that the Shareshian-Wachs conjecture should hold for generalized Hessenberg functions. 
\begin{theorem}[Generalized Shareshian-Wachs conjecture]\label{86} 
For any generalized Hessenberg function $h\in \cH_{I,n}$,  we have the equality 
\[\csf_q(h)=\omega\cF(h)\]
where $\omega$ is the involution of $\Lambda[q]$ sending a Schur function to its transpose. 
\end{theorem}
\begin{proof}
Let $\tilde{h}\in \cH_n$ be the Hessenberg function defined by \eqref{28} where $I$ is written in the form of \eqref{27}. By Theorem \ref{62}, we have $\csf_q(\tilde{h})=\omega \cF(\tilde{h})$. By Corollary \ref{31}, we find that the theorem follows directly from 
\beq\label{80d} \csf_q(\tilde{h})=\csf_q(h) \prod_{k=1}^{r+1}[i_k-i_{k-1}]_q!.\eeq
It is immediate that this equation follows from repeated applications of Lemma \ref{80} below. 
\end{proof}

\begin{lemma}\label{80}
Let $h\in \cH_{I,n}$ with $I=\{i_1,\cdots,i_r\}$ in the form of \eqref{27}. Let $1\leq k\leq r+1$, and let $I'=I\cup \{i_{k-1}+1, \cdots, i_k-1\}$. Let $h'\in \cH_{I',n}$ be 
\[h'|_I=h \quad \text{ and }\quad h'(i)=h(i_k) ~\text{ for }~ i_{k-1}<i\leq i_k.\]
Then we have $\csf_q(h')=[i_k-i_{k-1}]_q!\, \csf_q(h)$.
\end{lemma}
\begin{proof}
Let $m:=i_k-i_{k-1}.$ If $m=1$, there is nothing to prove. So we may assume $m>1$.  
The graph of $\Gamma_{h'}$ is obtained by replacing the vertex $\mathbf{i}_k$ by $m$ vertices 
and we have a map $\varphi:V(\G_{h'})\to V(\G_h)$ collapsing these $m$ vertices to $\mathbf{i}_k$. 
For $i<j$ in $I'\cup \{n\}$ with $\mathbf{i}$ or $\mathbf{j}$ not lying in the interval $(i_{k-1},i_k]$, there is an edge from $\mathbf{i}$ to $\mathbf{j}$ in $\Gamma_{h'}$ if and only if 
there is an edge from $\varphi(\mathbf{i})$ to $\varphi(\mathbf{j})$. 
For $\mathbf{i}$ and $\mathbf{j}$ in the interval $(i_{k-1},i_k]$, there is always an edge. 

A proper coloring $\gamma'$ of $(\G_{h'},w_{h'})$ induces a proper coloring $\gamma$ of $(\G_h, w_h)$ by $\gamma(\mathbf{i}_k)=\cup_{\mathbf{j}\in\varphi^{-1}(\mathbf{i}_k)}\gamma'(\mathbf{j}).$ A proper coloring $\gamma$ of $(\G_h,w_h)$ together with a total order on the set $\gamma(\mathbf{i}_k)$ defines a proper coloring of $(\G_{h'},w_{h'})$ uniquely. 
Therefore the difference between $\csf_q(h)$ and $\csf_q(h')$ comes from total orderings of the set $\gamma(\mathbf{i}_k)$ for each proper coloring $\gamma$ of $\G_h$, exactly as in the cases of \eqref{90a} and \eqref{90}. 
With these observations, it is straightforward to check that $[m]_q!\csf_q(h)=\csf_q(h')$. 
\end{proof}

Direct consequences of Theorem~\ref{86} are the positivity and the unimodality of $\csf_q(h)=\omega\cF(h)$ in the Schur basis.

\begin{corollary}
Let $h$ be a generalized Hessenberg function. Then the coefficients of $\csf_q(h)$ in the Schur basis are unimodal polynomials with non-negative coefficients.
\end{corollary}
\begin{proof}
By Theorem~\ref{86}, it is enough to prove the same statement for $\cF(h)$, since the involution $\omega$ preserves the Schur basis. The non-negativity is obvious for $\cF(h)$ since $A^*(X(h))$ is a genuine $S_n$-representation in each degree. 

To see the unimodality, let $h\in \cH_{I,n}$ and consider the line bundle $L:=\bigotimes_{i\in I}\det \cV_i^*$ on $X_h$, which is very ample. The first Chern class homomorphism $c_1(L):A^k(X_h)\to A^{k+1}(X_h)$ by $L$ is $S_n$-equivariant as in Proposition~\ref{24}. The unimodality follows from the hard Lefschetz property.
\end{proof}

The positivity part is independently proved in \cite{Hwa} by combinatorial means.

The following question seems also very interesting. 

\begin{question}[Generalized Stanley-Stembridge conjecture]\label{81}
Is $\csf_q(h)$ $\mathsf{e}$-positive for an arbitrary generalized Hessenberg function $h$? Equivalently, is the equivariant Poincar\'e polynomial $\cF(h)$ $\mathsf{h}$-positive for any $h$, where $\mathsf{h}$ denotes the complete homogeneous basis?
\end{question}

Note that when $h$ is an ordinary Hessenberg function, Question \ref{81} is the Stanley-Stembridge conjecture in \cite{SS}. 

The following is a generalization of \cite[Conjecture 5.3]{AN}.
\begin{question}[Log-concavity]
Are the coefficients of $\csf_q(h)$ in the elementary basis log-concave polynomials for every generalized Hessenberg function $h$?
\end{question}

\bigskip

\section{New algorithms for the $S_n$-characters} \label{91}
In this section, we provide two new algorithms (Propositions \ref{101} and \ref{105}) for computation of the equivariant Poincar\'e polynomials $\cF(h)$ of generalized Hessenberg varieties $X_h$. Both are based on 
Corollary \ref{48} and Proposition \ref{73}. 
As an application, we provide a combinatorial formula (Theorem \ref{102}) for 
the equivariant Poincar\'e polynomial $\cF(h)$. 
Our algorithms for $\cF(h)$ are completely different from the previously known algorithm in \cite{AN} and can be easily implemented on computer programs like Sage. Codes for our algorithms on Sage will be available upon request.

\medskip

By Corollary \ref{31}, the equivariant Poincar\'e polynomials $\cF(h)$ for generalized Hessenberg functions $h$ are readily obtained once we know those for ordinary Hessenberg functions. 
We say a generalized Hessenberg function $h$ is \emph{irreducible} if $X_h$ is. 

\subsection{Algorithm 1: reduction to the permutohedral varieties} \label{103}
The first algorithm is based on Corollary~\ref{78} and Example~\ref{50}. Geometrically, this is the reduction of $X_h$ to the permutohedral variety $X_{h_1}$ by the birational geometry described in Proposition \ref{73}.

Let $h\in \cH_n$ be irreducible, so that $h(i)>i$ for $i<n$ and hence $h\ge h_1$. 
By Corollary \ref{78}, the rational map
\[X_h \dashrightarrow X_{h_1},\]
which generically sends a complete flag $(V_i)_{i\in [n]}$ to another complete flag $(W_i)_{i\in[n]}$ defined by
\[W_{n-i}=V_{n-1}\cap x^{-1}V_{n-1}\cap \cdots \cap x^{-i+1}V_{n-1},\qquad 1\leq i<n\]
is a composition of the modifications in Proposition \ref{73}. 
Hence $\cF(h)$ can be computed by applying \eqref{75} and \eqref{76} repeatedly until we reach $\cF(h_1)=Q_n$ for which we can apply \eqref{55}. 
So we have the following. 

\begin{proposition}[Algorithm 1] \label{101}
Let $h\in \cH_n$. Let $p\geq0$. The $S_n$-representation on $A^p(X_h)$ is computed as follows.
\begin{enumerate} 
    \item[Step 0.] The algorithm is based on induction on $n$ and degree $p$. We may assume that the representations on the cohomology of Hessenberg varieties with smaller $n$ or $p$ are already known.
    \item[Step 1.] If $h(i)=i$ for some $i<n$, then use Proposition \ref{34}.
     Return to Step 0 for $h'$ and $h''$.
    \item[Step 2.] If $h$ is irreducible and $h(i)=h(i+1)$ for some $i< n-1$, then apply Corollary~\ref{78} to reach $h_1$.
    \begin{enumerate}
        \item All the blowup centers contribute to $A^p(X_h)$ with classes of  degree $<p$. Return to Step 0 for the blowup centers. 
        \item The representation on $A^*(X_{h_1})$ is known by \eqref{55}.
    \end{enumerate}
\end{enumerate}
\end{proposition}

Recall that $Q_n:=\cF(h_1)$ for $h_1\in \cH_n$ and $Q_1=1$.  
Let $Q_\lambda:=Q_{\lambda_1}\cdots Q_{\lambda_l}$ for each partition $\lambda=(\lambda_1,\cdots, \lambda_l)$ of $n$. 
As a result of Algorithm 1, we obtain $\cF(h)$ in the form of
\[\cF(h)=\sum_{\lambda\vdash n} d_{\lambda}Q_\lambda \]
where $d_\lambda\in \mathbb{Z}[q]$.
We can write down $d_\lambda$ more explicitly. Let $h\in \cH_n$ be an irreducible Hessenberg function. 
\begin{definition}
A sequence of Hessenberg functions $(h^{(0)},\cdots, h^{(\ell)})$ in $\cH_n$ is called \emph{admissible} if 
 $h^{(\ell)}\leq h_1$, and for $0\leq r<\ell$,  letting 
   $$i_r=\min\{1\leq i<n~:~h^{(r)}(i)=h^{(r)}(i+1)\neq i+1\}$$
   so that $\lvert(h^{(r)})^{-1}(i_r)\rvert\le 1$, 
        \[h^{(r+1)}=\begin{cases}\tau_{i_r}h^{(r)}\quad \text{ or }\quad  \tau_{i_r,i_r+1}h^{(r)}  &\text{ if } (h^{(r)})^{-1}(i_r)=\emptyset \\
        \tau_{i_r}^{i'_r}h^{(r)},\quad \tau_{i'_r,i_r}h^{(r)}\quad \text{ or }\quad \tau_{i_r,i_r+1}h^{(r)} &\text{ if } (h^{(r)})^{-1}(i_r)=\{i'_r\}.\end{cases}\]
\end{definition}
\begin{definition} Let $\overrightarrow{h}=(h^{(0)},\cdots, h^{(\ell)})$ be an admissible sequence of Hessenberg functions.
\begin{enumerate}
    \item We define a polynomial $w_{\overrightarrow{h}}\in \mathbb{Z}[q]$ by
    \[w_{\overrightarrow{h}}(q):=\prod_{r=0}^{\ell-1}w_{\overrightarrow{h},r}(q)\]
    where for each $0\leq r<\ell$, $w_{\overrightarrow{h},r}\in \mathbb{Z}[q]$ is defined by
    \[w_{\overrightarrow{h},r}=\begin{cases}1 &\text{ if } h^{(r+1)}=\tau^{i'_r}_{i_r}h^{(r)}\\
    1+q &\text{ if } h^{(r+1)}=\tau_{i_r}h^{(r)}\\
    q &\text{ if } h^{(r+1)}=\tau_{i'_r,i_r}h^{(r)}\\
    -q &\text{ if } h^{(r+1)}=\tau_{i_r,i_r+1}h^{(r)}.
    \end{cases}\]
    When $\ell=0$, define $w_{\overrightarrow{h}}(q):=1$.
    \item Let $\lambda(\overrightarrow{h})$ be the partition of $n$ defined by
    \[\lambda(\overrightarrow{h}):=(j_1,j_2-j_1,\cdots, j_k-j_{k-1})\]
    where $\{j:h^{(\ell)}(j)=j\}=\{j_1,\cdots,j_k\}$ with $j_1<\cdots<j_k=n$ and $j_0:=0$. In particular, if $h^{(\ell)}=h_1$, then $\lambda(\overrightarrow{h})=(n)$ is trivial.
\end{enumerate}
\end{definition}

Our proof of Corollary \ref{77} and Corollary \ref{78} together with Proposition \ref{101} give us the following. 
\begin{theorem}\label{102}
Let $\cA_h$ be the set of admissible sequences of Hessenberg functions starting with $h$, i.e. $h^{(0)}=h$. Then,
\[\cF(h)=\sum_{\overrightarrow{h}\in \cA_h}w_{\overrightarrow{h}}(q) Q_{\lambda(\overrightarrow{h})} \in \Lambda^n[q].\]
\end{theorem}

For $h\in \cH_n$, we write $h=(h(1),\cdots, h(n))$ and $\cF(h)=\cF(h(1),\cdots, h(n))$.
\begin{example}\label{79a}
	Let $h=(3,3,3)\in \cH_3$. By Algorithm 1, we have a modification $X_h\dasharrow X_{(2,3,3)}=X_{h_1}$ in Proposition~\ref{73}(1) at 1. By \eqref{75}, 
	\[\cF(h)=[2]_q\cF(2,3,3)-q\cF(2,2,3)=[2]_qQ_3-qQ_{(2,1)}=[3]_q!\mathsf{h}_3\]
	since $\cF(h_1)=Q_3=[3]_q\mathsf{h}_3+q\mathsf{h}_{(2,1)}$, $Q_2=[2]_q\mathsf{h}_2$ and $Q_1=\mathsf{h}_1$.
\end{example}
\begin{example}\label{79b}
	Let $h=(3,4,4,4)\in \cH_4$. By Algorithm 1, we have two modifications
	\[X_h\dasharrow X_{(3,3,4,4)}\dasharrow X_{(2,3,4,4)}=X_{h_1}\]
	in Proposition~\ref{73}(1) at 2 and 1 respectively. By \eqref{75}, 
	\begin{equation*}
		\begin{split}
			&\cF(h)=[2]_q\cF(3,3,4,4)-q\cF(3,3,3,4),\\
			&\cF(3,3,4,4)=[2]_q\cF(h_1)-q\cF(2,2,4,4)
		\end{split}
	\end{equation*}
	where $\cF(3,3,3,4)=\cF(3,3,3)Q_1=[2]_qQ_{(3,1)}-qQ_{(2,1,1)}$ by Example~\ref{79a} and $\cF(2,2,4,4)=Q_{(2,2)}$. Hence,
	\begin{equation*}
		\begin{split}
			\cF(h)&=[2]_q^2Q_4-q[2]_qQ_{(3,1)}-q[2]_qQ_{(2,2)}+q^2Q_{(2,1,1)}\\
			&=[4]_q[2]_q^2\mathsf{h}_4+q^2[2]_q\mathsf{h}_{(3,1)}
		\end{split}
	\end{equation*}
	since $\cF(h_1)=Q_4=[4]_q\mathsf{h}_4+q[2]_q\mathsf{h}_{(3,1)}+q[2]\mathsf{h}_{(2,2)}$.
\end{example}
\begin{example}\label{79c}
	Let $h=(3,4,5,5,5) \in \cH_5$. We have four modifications
	\[X_h\dasharrow X_{(4,4,4,5,5)}\dasharrow X_{(3,4,4,5,5)}\dasharrow X_{(3,3,4,5,5)}\dasharrow X_{(2,3,4,5,5)}=X_{h_1}\]
	at 3, 1, 2 and 1 respectively. By \eqref{75} and \eqref{76},
	\begin{equation*}
		\begin{split}
			&\cF(h)=\cF(4,4,4,5,5)-q\cF(3,4,4,4,5)+q\cF(2,4,4,5,5),\\
			&\cF(4,4,4,5,5)=[2]_q\cF(3,4,4,5,5)-q\cF(3,3,4,5,5),\\
			&\cF(3,4,4,5,5)=[2]_q\cF(3,3,4,5,5)-q\cF(3,3,3,5,5),\\
			&\cF(3,3,4,5,5)=[2]_q\cF(h_1)-q\cF(2,2,4,5,5)
		\end{split}
	\end{equation*}
	where by Examples~\ref{79a} and \ref{79b} 	
	\begin{equation*}
		\begin{split}
			\cF(3,4,4,4,5)=&\cF(3,4,4,4)Q_1\\
						=&[2]_q^2Q_{(4,1)}-q[2]_qQ_{(3,1,1)}-q[2]_qQ_{(2,2,1)}+q^2Q_{(2,1,1,1)},\\
			\cF(3,3,3,5,5)=&\cF(3,3,3)Q_2=[2]_qQ_{(3,2)}-qQ_{(2,2,1)},\\
			\cF(2,2,4,5,5)=&Q_2\cF(2,3,3)=Q_{(3,2)}.
		\end{split}
	\end{equation*}
	Hence it suffices to compute $\cF(2,4,4,5,5)$.
	From the modification
	\beq\label{122}X_{(2,4,4,5,5)}\dasharrow X_{(3,3,4,5,5)}\eeq
	at 2, we have
	\begin{equation}\label{122b}
		\begin{split}
			\cF(2,4,4,5,5)&=\cF(3,3,4,5,5)-q\cF(2,3,3,5,5)+q\cF(1,3,4,5,5)\\
						&=\cF(3,3,4,5,5)-qQ_{(3,2)}+qQ_{(4,1)}.
		\end{split}
	\end{equation}
	Combining all these, we have
	\begin{equation*}
		\begin{split}
			\cF(h)=&[2]_q^3Q_5+(-q[2]_q^2+q^2)Q_{(4,1)}+(-2q[2]_q^2-q^2)Q_{(3,2)}\\
			&+q^2[2]_qQ_{(3,1,1)}+2q^2[2]_qQ_{(2,2,1)}-q^3Q_{(2,1,1,1)} \\
			=&(q^7+4q^6+7q^5+8q^4+8q^3+7q^2+4q+1)\mathsf{h}_5\\
				&+(2q^5+4q^4+4q^3+2q^2)\mathsf{h}_{(4,1)}+(q^4+q^3)\mathsf{h}_{(3,2)}
		\end{split}
	\end{equation*}
	where $Q_5=[5]_q\mathsf{h}_5+q[3]_q\mathsf{h}_{(4,1)}+q([3]_q+[2]_q^2)\mathsf{h}_{(3,2)}+q^2\mathsf{h}_{(2,2,1)}$.
\end{example}

\subsection{Algorithm 2: reduction to the projective spaces}\label{104} 
The second algorithm is based on  Corollary~\ref{48}, Proposition \ref{73} and Corollary~\ref{77}. Geometrically this is reduction to projective spaces. One big difference from Algorithm 1 is that this essentially involves the generalized Hessenberg varieties of the form $X_h$, $h\in \cH_{r,n}$.

For each $h\in \cH_{r,n}$, consider the rational map
\[X_h \dashrightarrow \PP^{n-1}, \quad (V_i)_{i\in[r]}\mapsto V_1\]
which generically forgets all the subspaces except $V_1$. By Corollary~\ref{48} and Corollary~\ref{77}, this rational map can be resolved by the blowups in Corollary~\ref{48} and the modifications in Proposition \ref{73}.
So we have the following.

\begin{proposition}[Algorithm 2]\label{105} 
Let $h\in \cH_{r,n}$. Let $p\geq 0$. The $S_n$-representation on $A^p(X_h)$ is computed as follows.
\begin{enumerate}
    \item[Step 0.] The algorithm is based on the induction on $n$ and degree $p$. We may assume that the representations on the cohomology of generalized Hessenberg varieties with smaller $n$ or $p$ and arbitrary $r$ are already known.
    \item[Step 1.] Fix $n$ and $p$. Then the algorithm uses the induction on $r$. Hence we further assume that the representations on $A^p(X_h)$ for $h\in \cH_{r',n}$ are already known for all $r'<r$.
    \item[Step 2.] If $h(r)<n$, then $h(r)=r$, hence by Proposition \ref{34},
    \[\ch(A^p(X_h))=\ch(A^p(X_{h|_{[r]}}))\mathsf{h}_{n-r}\]
    with $h|_{[r]}\in \cH_r$. 
 	Return to Step 0 for $h|_{[r]}\in \cH_r$.
 	\item[Step 3.] 
   	Suppose $r>1$ and $h(r)=n$.
    \begin{enumerate}
        \item[(1)] If $\lvert h^{-1}(r)\rvert \leq 1$, then apply Corollary~\ref{48}. The base $X_{\kappa_rh}$ of $\rho_r$ is handled by $\kappa_rh\in\cH_{r-1,n}$. Return to Step 1 for the base. 
        
        	In case $\lvert h^{-1}(r)\rvert=1$, the blowup center contributes from degrees smaller than $p$.  
        	Return to Step 0 for the blowup center.
        \item[(2)] If $\lvert h^{-1}(r)\rvert\geq2$, then apply Corollary \ref{77}.
            The resulting $h'$ satisfies $|{h'}^{-1}(r)|\le 1$. Return to Step 3-(1).
            
            All the blowup centers which appear in the course of application of Corollary~\ref{77} contribute from degrees smaller than $p$. Return to Step 0 for the blowup centers.
    \end{enumerate}
    \item[Step 4.] Suppose $r=1$ and $h(1)=n$. Then %the resulting generalized Hessenberg variety is 
    $X_h=\PP^{n-1}$, and the representations on its cohomology are trivial, as we have seen in \eqref{23b}.
\end{enumerate}
\end{proposition}

    The output of Algorithm 2 is of the form $\sum_{\lambda\vdash n}c_\lambda \mathsf{h}_\lambda$ for some $c_\lambda \in \Z$, where $\mathsf{h}_\lambda$ are the complete homogeneous symmetric polynomials.

\medskip

Examples~\ref{121} and \ref{50} implement Algorithm 2. We provide another example below, in which we compute $\cF(2,4,4,5,5)$ using Algorithm 2.
One can compare this with \cite[Example 2.9]{AN}.

\begin{example}\label{125}
	Let $h=(2,4,4,5,5)$. We have
	\[X_h\dashrightarrow X_{(3,3,4,5,5)}\xrightarrow{~\rho_4~}X_{h'}\]
	where the first rational map is the modification in \eqref{122}. By Corollary~\ref{48}(2), the forgetful map $\rho_4$ is the blowup of $X_{h'}$ along $X_{h''}$ where $h'=\kappa_4(3,3,4,5,5)$ and $h''=\kappa_4\tau_3(3,3,4,5,5):\{1,2,3,5\}\to \{1,2,3,5\}$ are given by 
	\[\begin{split}
		&(h'(1),h'(2),h'(3),h'(5))=(3,3,5,5)\\
		&(h''(1),h''(2),h''(3),h''(5)=(3,3,3,5).
	\end{split}\]
	In particular, by \eqref{74} and Proposition~\ref{34},
	\[\cF(3,3,4,5,5)=\cF(h')+q\cF(h'')=\cF(h')+q[3]_q!\mathsf{h}_{(3,2)}\]
	where the second equality reflects the fact that $X_{h''}$ is the disjoint union of $\binom{5}{3}$ copies of $\Fl(3)$. By \eqref{76}, which reads as \eqref{122b} in this case, we have
	\beq\label{123a}
	\begin{split}
		\cF(h)&=\cF(3,3,4,5,5)+q[4]_q\mathsf{h}_{(4,1)}-q[3]_q!\mathsf{h}_{(3,2)}+q^2[2]_q\mathsf{h}_{(3,1,1)}\\
		&=\cF(h')+q[4]_q\mathsf{h}_{(4,1)}+q^2[2]_q\mathsf{h}_{(3,1,1)}.
	\end{split}
	\eeq 
	
	To compute $\cF(h')$, consider the modification at 1
	\[X_{h'}\dashrightarrow X_{\tau_1h'}\]
	where $\tau_1h'$ is given by $(\tau_1h(1),\tau_1h(2),\tau_1h(3),\tau_1h(5))=(2,3,5,5)$. Since $\tau_1h'$ is equal to $h^{(3)}$ in Example~\ref{50} with $n=5$, we have  
	by \eqref{54}
	\beq \label{123c}\begin{split}
		\cF(\tau_1h')=\cF(h^{(3)})&=\cF(h^{(2)})+q[2]_q^2\mathsf{h}_{(3,2)}\\
		&=\cF(h^{(1)})+q[3]_q\mathsf{h}_{(4,1)}+q[2]_q^2\mathsf{h}_{(3,2)}\\
		&=[5]_q\mathsf{h}_{5}+q[3]_q\mathsf{h}_{(4,1)}+q[2]_q^2\mathsf{h}_{(3,2)}.
	\end{split}\eeq
	Note that \eqref{54}, hence \eqref{123c}, are computation by Algorithm 2.
	By \eqref{75},  
	\beq \label{123b}\begin{split}
		\cF(h')&=[2]_q\cF(\tau_1h')-q\cF(\tau_{1,2}h')\\
		&=[2]_q([5]_q\mathsf{h}_{5}+q[3]_q\mathsf{h}_{(4,1)}+q[2]_q^2\mathsf{h}_{(3,2)})-q[3]_q[2]_q\mathsf{h}_{(3,2)}\\
		&=[2]_q([5]_q\mathsf{h}_{5}+q[3]_q\mathsf{h}_{(4,1)}+q^2\mathsf{h}_{(3,2)})
	\end{split}\eeq
	where the second equality is given by \eqref{123c} and 
	the fact that $X_{\tau_{1,2}h'}$ is the disjoint union of $\binom{5}{2}$ copies of $\Fl(2)\times \Fl_{\{1\}}(3)\cong \PP^1\times \PP^2$.
	
	Combining \eqref{123a} and \eqref{123b}, we have
	\[\cF(h)=[5]_q[2]_q\mathsf{h}_{5}+q([4]_q+[3]_q[2]_q)\mathsf{h}_{(4,1)}+q^2[2]_q\mathsf{h}_{(3,2)}+q^2[2]_q\mathsf{h}_{(3,1,1)}.\]
\end{example}
\medskip

\subsection{Multiplicities of trivial representations}
We compute the multiplicities of the trivial representation in the expansions in $\mathsf{h}_\lambda$ and in the Schur basis respectively. Our algorithms allow another simple proofs of the following formulas in \cite{SW,AHHM,ST}. 
\begin{proposition}\label{106} 
Let $h\in \cH_{r,n}$.
\begin{enumerate}
    \item The coefficient of $\mathsf{h}_{n}$ in the expansion of $\cF(h)$ in $\{\mathsf{h}_\lambda\}_{\lambda\vdash n}$ is
    \[c_{(n)}=[n]_q\prod_{i=1}^{r-1}[h(i)-i]_q,\]
    for an irreducible $h$.
    \item The dimension of the $S_n$-invariant part is
    \[\sum_{\lambda \vdash n}c_{\lambda}=\prod_{i=1}^r[h(i)-i+1]_q.\]
\end{enumerate}
\end{proposition}
\begin{proof}
We prove these using Algorithm 2. When $X_h=\PP^{n-1}$, they trivially hold, since in this case $\cF(h)=[n]_q$ with $r=1$. Hence it suffices to show that the formulas are compatible with the relations given in 
Corollary~\ref{48} and Proposition \ref{73}. 
It can be immediately checked that they are indeed compatible in (1) and (2), using the identity
\[[a+1]_q[b+1]_q=[a+b+1]_q+q[a]_q[b]_q\]
for $a,b\geq 0$. We omit the detail.
\end{proof}
\begin{remark}
 For $h\in \cH_n$, the above formulas match with \cite[Theorem 7.1 and Theorem 6.9]{SW} respectively via the involution $\omega$ in Theorem \ref{62}. On the geometry side, the formula (2) was proved in  \cite{AHHM} by constructing a ring isomorphism $A^*(X_h)^{S_n}\cong A^*(X_h^\mathfrak{n})$, where $X_h^\mathfrak{n}$ denotes the regular nilpotent Hessenberg variety associated to $h\in \cH_n$. The Poincar\'e polynomial of $X_h^\mathfrak{n}$ was computed in \cite{ST}.
\end{remark}

\medskip

By Algorithm 2, we can also compute other coefficients. 
\begin{example}[\cite{SW,AHM}] \label{ex1}
For $1\leq r <n$, let $h^{(r)}\in \cH_{r,n}$ be defined by $h^{(r)}(1)=r$ and $h^{(r)}(i)=n$ for $2\leq i\leq r$. By Corollary~\ref{48}, the forgetful morphism $X_{h^{(r)}}\to \Fl_{r-1}(n)$ is the blowup along $X_{h^{(r-1)}}$ for $r\geq 2$, and $X_{h^{(1)}}$ is the set of coordinate $n$ points in $\PP^{n-1}$. In particular, we have 
\[\cF(h^{(r)})=\frac{[n]_q!}{[n-r+1]_q!}\mathsf{h}_n +q[n-r]_q\cF(h^{(r-1)})\]
for $r\geq 2$, and $\cF(h^{(1)})=\mathsf{h}_{(n-1,1)}$. From this, it follows that
\[\cF(h^{(r)})=q^{r-1}\frac{[n-2]_q!}{[n-r-1]_q!}\mathsf{h}_{(n-1,1)} \quad \text{ modulo } ~\mathsf{h}_n.\]

We also remark that the above description allows a geometric basis which is $S_n$-invariant and the stabilizer of each element of which is isomorphic to the Young subgroup $S_{n-1}\times S_1$ or  $S_n$.
\end{example}

\bigskip

\section{Degree $k+1$ for $h_k$}\label{110}
In this section, we provide an alternative proof of some recent results on the representations in low degrees  
in \cite{AMS} by our algorithms. Using these and Algorithm 2, we compute the $S_n$-representations on the $(k+1)$-st cohomology $A^{k+1}(X_{h_k})$ of the Hessenberg variety $X_{h_k}$. (See Definition~\ref{84} for $h_k$.) 

We denote by $M^\lambda$ the permutation module of $S_n$ corresponding to a partition $\lambda$ of $n$. Recall that $\ch~ M^\lambda=\mathsf{h}_\lambda$.

\begin{theorem} \label{111}
Let $n\geq k+2$.
\begin{enumerate}
    \item When $k=2$,
  \[  \begin{split}A^3(X(h_2))=&\sum_{i=0}^3\binom{n-2}{i}M^{(n)}+\Big(n^2-7n+14\Big)M^{(n-1,1)}\\
    &+(n-4)M^{(n-2,2)}+\sum_{i=3}^{n-3}M^{(n-i,i)}.\end{split}\]
    \item When $k\geq 3$,
    \[A^{k+1}(X(h_k))=aM^{(n)}+\Big(n^2-(k+3)n+2k+1\Big)M^{(n-1,1)},\]
    where the coefficient $a$ of $M^{(n)}$ is determined by any of the following two ways:
    \begin{enumerate}
        \item $[n]_q[k-1]_q![k]_q^{n-k}=\cdots + aq^3+\cdots$, or
        \item $[k]_q![k+1]_q^{n-k}=\cdots + \Big(a+n^2-(k+3)n+2k+1\Big)q^3+\cdots$.
    \end{enumerate}
\end{enumerate}
\end{theorem}

The equalities (a) and (b) are from Proposition~\ref{106} (1) and (2) respectively.

\medskip
The rest of this section is devoted to a proof of Theorem \ref{111}. 
In the course of the proof, we use several known results as well as Algorithm 2. In the first two subsections, we reprove the results in a recent paper \cite{AMS} on the representations on low degree cohomology, using the algorithms in \S\ref{91}. Then, we prove Theorem \ref{111} in the last subsection.

\subsection{Degree 1} For degree 1, the representations are completely known. Let $h\in \cH_{r,n}$.

\begin{theorem}\cite[Theorem 6.4]{CHL} \cite[Theorem 5.2]{AMS}  \label{114}
Let $h\in \cH_{r,n}$ be irreducible. For $1\leq i<r$, let
\[\beta_i:=\begin{cases}(n-i,i) &\text{ if } h(i-1)=i,~h(i)=i+1,\\
            (n-1,1) & \text{ if }h(i-1)=h(i)=i+1\\
            (n) & \text{ otherwise}, \end{cases}\]
where we set $h(0)=1$. Then we have the equality of $S_n$-representations
\[A^1(X_h)=M^{(n)}+\sum_{i=1}^{r-1}M^{\beta_i}.\]
\end{theorem}
\begin{proof} We use a version of Algorithm 1 for generalized Hessenberg varieties. By the same argument in the proof of Corollary~\ref{78}, one can modify a given irreducible $h\in \cH_{r,n}$ to  
$\k_{>r}h_1 \in \cH_{r,n}$ where $\k_{>r}h_1=\kappa_{r+1}\cdots\kappa_{n-1}h_1$.

Since the assertion holds for $\k_{>r}h_1$ by observing that $A^1(X_{\k_{>r}h_1})=\sum_{i=1}^{r-1}M^{(n-i,i)}$ as in Example~\ref{50}, it suffices to show that the assertion is compatible with \eqref{75} and \eqref{76}. 
This can be checked by simple computation. 
\end{proof}

\subsection{Pullback morphisms from $\Fl_r(n)$} 
In this subsection, we prove another result of \cite{AMS} on the representations $A^p(X_h)$ in low degrees, under some conditions.
\begin{theorem} \cite[Theorem 6.1, Corollary 6.2]{AMS} \label{113}
Let $h\in \cH_{r,n}$ and $k\geq 2$. Suppose that $h(i)\geq \min\{i+k,n\}$ for $i\in [r]$. 
    \begin{enumerate}
        \item for each $p<k$, the pullback
    \[A^{p}(\Fl_r(n))\longrightarrow A^{p}(X_h)\]
    by the natural inclusion is an isomorphism, and
        \item when $p=k$,  
        \[A^{k}(X_h)=A^k(\Fl_r(n))+m\Big(M^{(n-1,1)}-M^{(n)}\Big),\]
    where $m=|\{i \in [r]~:~ h(i)=i+k<n\}|$.
\end{enumerate}
\end{theorem}
\begin{proof}
We prove this by using a combination of Corollary~\ref{78} and Algorithm 1. Although Corollary~\ref{78} states for $h\in \cH_n$ and $h_k$, the same holds for $h\in \cH_{r,n}$ and $\k_{>r}h_k$.

By abuse of notation, we write $h_k$ also for $\k_{>r}h_k$. There must be no confusion since $r$ remains fixed until we use Algorithm 1.

The proof is divided into two steps. Firstly, we reduce the assertions to the case $h=h_k$ by Corollary~\ref{78} and induction on $k$. Then we prove the assertions for $h_k$ using Algorithm 1.

(1) The assertion actually holds for $k=1$, since $A^0(X_h)=\mathbb{Q}$ for an irreducible $h$. By induction on $k$, we assume that the assertion holds up to $k-1$. 

Let us reduce the assertion to the case $h=h_k$. Key ingredients are Corollary~\ref{78} and Lemma~\ref{112} below. Suppose that $h$ is modified to $h'$, so $h'=\tau_jh$ or $\tau_j^{j_0}h$. If $h$ satisfies the inequality $h(i)\geq \min\{i+k,n\}$, then so does $h'$. Furthermore, if $h''=\tau_{j,j+1}h$ or $\tau_{j_0,j}h$, which corresponds to one of the blowup centers in the modifications, then $h''$ satisfies the inequality $h''(i)\geq \min\{i+k-1,n\}$ for $i\in [r]$, so the assertion holds for $h''$ by induction. Therefore, when $h'=\tau_{j}^{j_0}h$, by applying Lemma~\ref{112} below to the left and the right squares of the modification diagram
\[\xymatrix{\Fl_r(n)\ar[r] & \Fl_{[r]-\{j\}}(n) & \Fl_r(n)\ar[l]\\
X_h \ar[u] \ar[r]^{\rho_j}& X_{\kappa_jh}\ar[u] & X_{h'}\ar[l]_{\rho_j}\ar[u]}\]
where both $\rho_j$ are blowups, we conclude that the assertion holds for $h$ if and only if it holds for $h'$. 
When $h'=\tau_jh$ so that $\rho_j$ in the left square is $\PP^1$-bundle, applying Lemma~\ref{112} to the right square is enough to conclude that the assertion holds of $h$ if and only if it holds for $h'$, again by  induction. Repeating this, by Corollary~\ref{78}, it suffices to prove the assertion for $h_k$.

Now let $h=h_k$. If $k\geq r$, then $X_h=\Fl_r(n)$ for which the assertion holds. Assume that $k<r$ so that $h(1)=k+1\leq r$. Then by using Algorithm 1 and by applying Lemma~\ref{112} to the blowup diagram in Theorem~\ref{42}, we reduce the assertion to the case $X_h=\Fl_k(n)$, for which the assertion is obvious. By induction, we may assume that the assertion holds for the blowup centers. This proves (1).

(2) By Proposition~\ref{106}, it suffices to prove that the coefficient of $M^{(n-1,1)}$ in $A^k(X_h)$ is equal to the number of $i$ with $h(i)=i+k<n$, and that the coefficients of $M^\lambda$ with $\lambda\neq (n),(n-1,1)$ always vanish, in the expansion in $M^\lambda$.

One can check that the coefficients of $M^\lambda$ with $\lambda\neq (n)$ in the assertion are compatible with the formulas in Proposition~\ref{73}. In the course of reduction, one also needs Theorem~\ref{114} when $k=2$. As a result, we may assume $h=h_k$. If $k\geq r$, then $X_h=\Fl_r(n)$ for which the assertion holds. If $k<r$, then by Algorithm 1, $\rho_{k+1}\circ\cdots\circ\rho_r:X_h \to\Fl_k(n)$ is the composition of an iterated sequence of blowups. Since the blowup center contributes to the degree $k$ part with one copy of $M^{(n-1,1)}$ for each $\rho_j$, we find that $A^k(X_h)=(r-k)M^{(n-1,1)}$ modulo $M^{(n)}$. The assertion for $h_k$ follows by Proposition~\ref{106}.
\end{proof}

\begin{lemma} \label{112}
    Let $Z\subset X\subset P$ be smooth projective varieties. Let $\widetilde{X}=\mathrm{Bl}_ZX$ be the blowup of $X$ along $Z$ with exceptional divisor $E$. Let $\widetilde{P}$ be a projective bundle over $P$. Suppose that there exists a closed embedding $\tX\xrightarrow{}\tP$ over $P$ such that the big square of the commutative diagram 
    \[\xymatrix{E\ar[r] \ar[d]_-{p} & \tX \ar[r] \ar[d] & \tP \ar[d]^-{\pi} \\
    Z \ar[r] & X \ar[r] & P}\]
    is Cartesian. Then for $p>0$, any two of the following imply the other.
    \begin{enumerate}
        \item The pullback $A^i(\tP) \xrightarrow{} A^i(\tX)$ is an isomorphism for $i\leq p$.
        \item The pullback $A^i(P)\xrightarrow{}A^i(X)$ is an isomorphism for $i\leq p$.
        \item The pullback $A^i(\tP)\xrightarrow{}A^i(E)$ is an isomorphism for $i<p$.
    \end{enumerate}
\end{lemma}
\begin{proof} Note that $\sO_{\pi}(-1)|_E$ is isomorphic to $\sO_{\tX}(E)|_E$ modulo $\Pic Z$. Let $Z\subset X$ be of codimension $r+1$.
Then we have
\[A^*(\tX)\cong \pi|_{\tX}^*A^*(X) \oplus \bigoplus_{j=1}^r p^*A^*(Z)\Big(\cap ~\sO_{\tX}(E)|_E\Big)^{j},\]
while
\[A^*(\tP)\cong \pi^* A^*(P)\oplus \bigoplus_{j=1}^r \pi^*A^*(P)\Big(\cap \sO_\pi(-1)\Big)^{j}.\]
Since the pullback homorphisms map componentwisely, for $p>0$ any two of (1), (2) and ($3'$) below imply the other.
\begin{enumerate}
    \item[($3'$)]  The pullback $A^i(P)\xrightarrow{}A^i(Z)$ is an isomorphism for $i<p$.
\end{enumerate}
Since (3) and ($3'$) are equivalent, the assertion follows.
\end{proof}

\begin{remark}
Using the techniques in the proof of Theorem~\ref{113}, one can also prove the following: for $h,h'\in \cH_n$ abelian with $h\leq h'$, the pullback
\[A^*(X_{h'})\longrightarrow A^*(X_h)\]
by the canonical inclusion is an isomorphism in degrees up to
\[\min\{h(i)-i: h(i)<h'(i),~1\leq i<n\}-1.\]
\end{remark}

\subsection{Proof of Theorem~\ref{111}}
Thanks to Proposition~\ref{106}, it suffices to compute the coefficients of $M^\lambda$ with $\lambda\neq(n)$ in the expansion in $M^\lambda$. All the computation below is written modulo $M^{(n)}$.

\underline{Step 1}. By Algorithm 2, we have a sequence of iterated blowups
\[X_{h_k}\xrightarrow{~\rho_{n-1}~}~\cdots~ \xrightarrow{~\rho_{k+2}~}X_{\k_{>k+1}h_k}\xrightarrow{~\rho_{k+1}~}\Fl(k,n),\]
where $\k_{>k+1}h_k=\tau_{k+2}\cdots\tau_{n-1}h_k$.
As in Example~\ref{ex1},
\begin{equation}\label{eq:pf1}
    A^{k+1}(X_{\k_{>k+1}h_k})= kM^{(n-1,1)}.
\end{equation}
For $2\leq j \leq n-k-1$, denote the blowup center of $\rho_{n-j+1}$ by 
\[Z_j:=X_{\k_{>n-j}\tau_{n-j-k+1}h_k}\subset X_{\k_{>n-j}h_k},\]
where the codimension is $j$. By the blowup formula, 
\begin{equation}\label{eq:pf2}
    A^*(X_{h_k})=A^*(X_{\k_{>k+1}h_k})\oplus q\bigoplus_{j=2}^{n-k-1}[j-1]_q A^*(Z_j).
\end{equation}
So it suffices to compute $A^{\leq k}(Z_j)$.

\underline{Step 2}. Note that the (iterated) forgetful morphism
\[ X_{h'}\longrightarrow Z_j,\]
which forgets the subspaces of dimension $n-1,\cdots,n-j+1$ in each flag, is a $\Fl(j)$-bundle, where $h'$ is a Hessenberg function defined by $h'(n-j+1)=\cdots=h'(n-1)=n$ as in Theorem~\ref{4}. Let $W_j:=X_{(h')^t}$ be the Hessenberg variety whose corresponding Hessenberg function is the \emph{transpose} of $h'$, which is defined by
\[(h')^t(i)=n+1-\min\{j\in [n]:h'(j)\geq n-i+1\}\ \ \text{for } i\in [n].\]
We write $W_j^t:=X_{h'}$. By the canonical duality $(V_i)_{i\in[n]}\mapsto ((\C^n/V_{n-i})^*)_{i\in[n]}$, there exists an equivariant isomorphism (cf. \cite[Theorem 1.2]{AN})
\begin{equation}\label{eq:pf3}
    A^*(W_j)\cong A^*(W_j^t).
\end{equation}
Therefore, by the projective bundle formula,
\begin{equation}\label{eq:pf4}
    A^*(W_j)=[j]_q!A^*(Z_j).
\end{equation}
In order to compute $A^{\leq k}(Z_j)$, now it suffices to compute $A^{\leq k}(W_j)$.

\underline{Step 3}. The Hessenberg function $h^{W_j}$ of $W_j$ is defined by
\[h^{W_j}(i)=\begin{cases}j+k-1 &\text{ if } i\leq j\\
h_k(i) & \text{ if } i>j.\end{cases}\]
By applying Algorithm 1 to $W_j$, we have a sequence of iterated blowups
\[W_j \xrightarrow{~\rho_{n-1}~}~\cdots~ \xrightarrow{~\rho_{j+k}~} X_{\k_{>j+k-1}h^{W_j}}=:\overline{W}_j,\]
where
\begin{enumerate}
    \item $\overline{W}_j$ is the generalized Hessenberg variety corresponding to $h_{j,n-j-k+1}\in \cH_{j+k-1,n}$ using the notation in Lemma~\ref{119} below,
    \item $\rho_{j+k}$ is a $\PP^{n-j-k}$-bundle,
    \item $\rho_{l+k+1}$ is the blowup with the center \[W_j^{l+k}:=X_{\k_{>l+k}\tau_{l}h^{W_j}}\subset X_{\k_{>l+k}h^{W_j}}\]
    of codimension $n-l-k$, for each $j\leq l \leq n-k-2$. 
\end{enumerate}
By the blowup formula and the projective bundle formula, we have
\begin{equation}\label{eq:pf5}
    A^*(W_j)=[n-j-k+1]_qA^*(\overline{W}_j)\oplus q\bigoplus_{l=j}^{n-k-2}[n-l-k-1]_qA^*(W_j^{l+k}).
\end{equation}
Therefore, in order to compute $A^{\leq k}(W_j)$, it is enough to compute (i) $A^{\leq k}(\overline{W}_j)$ and (ii) $A^{\leq k-1}(W_j^{l+k})$. The latter is computed as
\begin{equation}\label{eq:pf6}
    A^{\leq k-1}(W_j^{l+k})= \begin{cases}qM^{(n-1,1)}+qM^{(n-j-1,j+1)} \quad \text{ if }k=2, l=j\\
    2q^{k-1}M^{(n-1,1)} \qquad \text{ if }k=2, j<l\leq n-4 \text{ or if }k\geq 3\end{cases}
\end{equation}
directly by Theorem~\ref{113}, since the corresponding Hessenberg functions contain $h_{k-1}$. On the other hand, the former can be computed inductively by Lemma~\ref{119}:
\begin{equation}\label{eq:pf7}
    A^{\leq k}(\overline{W}_j)=\begin{cases}(q+jq^2)M^{(n-1,1)}+q^2M^{(n-2,2)} & \text{ if }k=2\\
    (q^{k-1}+(j+k-1)q^k)M^{(n-1,1)}& \text{ if }k\geq 3\end{cases}
\end{equation}

As a combination of \eqref{eq:pf1}, \eqref{eq:pf2}, \eqref{eq:pf3}, \eqref{eq:pf4}, \eqref{eq:pf5}, \eqref{eq:pf6} and \eqref{eq:pf7}, we obtain the desired formula. This completes our proof of Theorem~\ref{111}.

\bigskip
The lemma below was used in the proof of Theorem~\ref{111}. 

\begin{lemma}\label{119}
    Let $l,m\geq 1$ be integers with $l+m\leq n$. Let $\cF_{l,m}:=\cF(h_{l,m})$ for $h_{l,m}\in \cH_{n-m,n}$ be the generalized Hessenberg function on $[n-m]$ defined by
    \[h_{l,m}(i)=\begin{cases} n-m & \text{ if }i\leq l\\
     n & \text{ if } l<i\leq n-m.\end{cases}\]
     Then, we have
     \[\cF_{l,m}=[l]_q\cF_{l-1,m+1}+q^l[m+1-l]_q\cF_{l,m+1}.\]
\end{lemma}
\begin{proof} This is a special case of \cite[Proposition 2.4]{AN}. One can also easily prove this using Theorem~\ref{42} and Proposition~\ref{73}(1).
\end{proof}

We end this paper with the following conjecture which is checked to be true for $n\leq 13$ %obtained 
by computer calculation via Algorithm 1. Let $\cF(h_2)=\sum_{\lambda \vdash n}c_\lambda \mathsf{h}_\lambda$ be the equivariant Poincar\'e polynomial of $X_{h_2}$.

\begin{conjecture} For $\lambda=(n-j,j)$, $c_\lambda$ are equal to the following.
    \begin{enumerate}
        \item When $j<\frac{n}{2}$, 
        \[q[2]_q^{n-2}\Big([n-j-1]_q[j]_q+[n-j]_q[j-1]_q\Big)+d_{(n-j,j)}[n-j]_q[j]_q,\]
        where $d_{(n-j,j)}$ is equal to
        \begin{equation*}
        \begin{split}
            &-2q[2]_q^{n-3}-2q^2[2]_q^{n-5}-\cdots -2q^{j-1}[2]_q^{n-2j+1}-q^j[2]_q^{n-2j-1}\\ 
            &+\max(n-2j-2,0)q^{j+1}[2]_q^{n-2j-3}.
        \end{split}
        \end{equation*}
        \item When $j=\frac{n}{2}$,
        \[q[2]_q^{n-2}\left[\frac{n}{2}-1\right]_q\left[\frac{n}{2}\right]_q+d_{(\frac{n}{2},\frac{n}{2})}\left[\frac{n}{2}\right]_q^2,\]
        where $d_{(\frac{n}{2},\frac{n}{2})}$ is equal to
        \[-q[2]_q^{n-3}-q^2[2]_q^{n-5}-\cdots -q^{\frac{n}{2}-1}[2]_q.\]
    \end{enumerate}
\end{conjecture}

For example, $c_{(n-1,1)}$ is conjecturally equal to
\[q[2]_q^{n-2}[n-2]_q-\Big(q[2]_q^{n-3}-(n-4)q^2[2]_q^{n-5}\Big)[n-1]_q.\]

\bibliographystyle{alpha}

\end{document}